\documentclass[options]{amsart}
\usepackage{latexsym,ifthen,amssymb}
\usepackage[toc,page,title,titletoc,header]{appendix}
\usepackage{enumerate}
\usepackage{graphicx}
\usepackage{bm}
\usepackage{subfigure}
\usepackage{float}
\usepackage{rotating}
\usepackage{epstopdf}
\usepackage{hyperref}
\usepackage{tikz-cd}

\usepackage[toc,page,title,titletoc,header]{appendix}
\usepackage{multirow}
 \usepackage{longtable}
 \usepackage{rotating}

\newtheorem{theorem}{Theorem}[section]
\newtheorem{lemma}[theorem]{Lemma}

\newtheorem{claim}[theorem]{Claim}

\newtheorem{proposition}[theorem]{Proposition}

\theoremstyle{definition}
\newtheorem{definition}[theorem]{Definition}

\newtheorem{example}[theorem]{Example}

\theoremstyle{remark}
\newtheorem{remark}[theorem]{Remark}
\newtheorem {question}[theorem]{Question}
\numberwithin{equation}{section}

\theoremstyle{noparens}
\newtheorem*{question*}{Question}
\newtheorem*{theorem*}{Theorem}

\newcommand{\uhr}{{\upharpoonright}}

\title{The combinatorial equivalence of a computability theoretic question}

\author{Lu Liu}
\address{Department of Mathematics and Statistics,
Central South University,
City Changsha, Hunan Province,
China. 410083}
\email{g.jiayi.liu@gmail.com}

\begin{document}

\maketitle

\def\opp{ENSH}
\def\PP{P}
\def\v{\vec}
\def\msf{\mathsf}
\def\t{\tilde}
\def\h{\hat}
\def\nr{n_r}
\def\2r{2_{r}}

\def\search{searching}
\def\sectionhomo{sectionally-homogeneous}

\begin{abstract}
We show that
a question of Miller and Solomon------that whether
there exists a coloring $c:d^{<\omega}\rightarrow k$
that does not admit a $c$-computable variable word infinite solution,
is equivalent to a natural, nontrivial combinatorial question.
The combinatorial question asked whether there is an infinite
sequence of integers such that each of its initial segment
satisfies a Ramsian type property.
This is the first computability theoretic question known to be equivalent to
a natural, nontrivial question that does not concern complexity notions.
It turns out that the negation of the combinatorial
question  is a generalization of
Hales-Jewett theorem. We solve some special cases of the combinatorial question
and obtain a generalization of Hales-Jewett theorem on some particular parameters.
\end{abstract}

\section{Introduction}

In this paper, we show that
a question of Miller and Solomon------that whether
there exists a coloring $c:d^{<\omega}\rightarrow k$
that does not admit a $c$-computable variable word infinite solution
(see Definition \ref{combdef2}),
is equivalent to a natural, nontrivial combinatorial question
(Theorem \ref{combth0}).
The combinatorial question asked whether there is an infinite
sequence of integers such that each of its initial segment
satisfies a Ramsian type property.
We should point out that it is not
unseen that a computability theoretic
question is equivalent to a non computability theoretic question.
However, all known examples are somewhat trivial in the following ways.
\begin{enumerate}
\item Given a computability theoretic
assertion, we can translate ``computability"
into its plain definition (say, the definition using Turing machine)
and obtain a long, unnatural combinatorial assertion.

\item The non computability theoretic question is trivial.
Every true computability theoretic proposition is, of course, equivalent to
the trivial assertion 1=1. In another word,
  once the equivalence relation is established, the computability theoretic question
is solved.

\item One of the directions of the equivalence is trivial.
In many cases, there exists an  object satisfying a property $\msf{P}$ implies there exists a
computable such object.
Usually, this is done by showing that there exists a computable tree $T$ in
$\omega^{<\omega}$ such that the   infinite paths through $T$
consist of the collection of objects satisfying $\msf{P}$,
and for every $\rho\in T$, $\rho$ admits some successor
in $T$.
It is trivial that there exists a computable $X\in \msf{P}$
implies that $\msf{P}\ne\emptyset$.

\end{enumerate}
In our example, both the combinatorial question and the
computability theoretic question are natural and nontrivial
(yet unsolved). Moreover, both directions of the equivalence
relation are nontrivial.

It turns out that disproving the combinatorial property on certain
sequences is equivalent to  Hales-Jewett theorem
(see Proposition \ref{comblem0}). Therefore, the negation of the combinatorial question,
namely every infinite sequence of integers
admits an initial segment that does not satisfy the combinatorial property is a
 generalization of Hales-Jewett theorem.
Hales-Jewett theorem states that given $d,k,n\in\omega$,
there exists an $N\in\omega$ such that for every coloring
$c:d^N\rightarrow k$, there exists an
$n$-dimensional combinatorial subspace   of $d^N$ that is
monochromatic for $c$
\footnote{In many references, this is referred to as multidimensional Hales-Jewett
theorem; and Hales-Jewett theorem refers to the conclusion
that there exists a $1$-dimensional combinatorial subspace that is monochromatic.}. An equivalent statement in terms of
the combinatorial property
(Proposition \ref{comblem0}) says that the
 combinatorial subspace can be chosen to satisfy certain property
 (the set of the smallest coordinates of each dimension can be very specific).
 Hales-Jewett theorem is of fundamental importance in combinatorics.
It implies
van der Waerden theorem (which states that for every partition of integers,
every $r\in\omega$, there exists an arithmetical progression of length $r$ in
one part); actually it implies the multidimensional van der Waerden theorem,
namely Gallai’s theorem.
Density  Hales-Jewett theorem
implies density   van der Waerden
theorem, namely Szemer\'{e}di's theorem, which states that
for every set $A$ of integers of positive density
(meaning $\limsup_{n\rightarrow\infty}|A\cap n|/n>0$), every $r\in\omega$,
there exists an arithmetical progression in $A$ of length $r$
(conjectured by Erd\H{o}s and Tur\'{a}n).
See \cite{dodos2016ramsey} for more details on Hales-Jewett theorem.
We prove that for every infinite sequence of positive integers
$n_0n_1\cdots$ with $n_0=2$, the combinatorial property does not
holds on $n_0\cdots n_r$ for some $r$ (Theorem \ref{combth2}). This generalizes
Hales-Jewett theorem on parameters: $d=2$, $k=2, n=2$.

\ \\

\noindent\textbf{Connections between
computability theory and combinatorics.}
Computability theory starts from Church-Turing thesis which clarifies what is computability
(what is an algorithm or effectiveness).
Early study of computability theory includes the complexity of certain objects such as
the set of integer-coefficient equations that admit integer roots, the set of algorithms that halt etc.
The concept of computability naturally generalizes to relative computability------$X$ computes $Y$ (that is $Y$ can be computed
using the information of $X$), and therefore gives rise
to the concept of Turing degree ($X,Y$ belong to the same Turing degree if and only if
they compute each other). Computability theory characterizes the complexity of various objects in terms
of Turing degree; and studies many structures related to the Turing degree and
the relation between different complexity notions.

Computability theory is rarely connected to other branch of mathematics. For example, it is now easy
to construct a computably enumerable (henceforth c.e.) set $X$ such that $X$ has Turing degree strictly between $\mathbf{0}$
and $\mathbf{0}'$ (where $\mathbf{0},\mathbf{0}'$ are the Turing degree of computable sets and the halting problem
respectively). However, we  do not know of a naturally defined (or non logically defined) such c.e. set.
Currently, there is no evidence that the collection of computable sets can be combinatorially or algebraically defined.
This situation is changing. Computability theory is having more and more connection to combinatorics.
To name a few:
\begin{enumerate}
\item Some complexity notions can be characterized in terms of some combinatorial notions. It is well known that
an oracle $D$ is hyperarithmetic if and only if there is a function $f\in\omega^\omega$
such that for every function $g\in\omega^\omega$ with $g(n)\geq f(n)  $ for all $n$, $g$ computes $D$.
Solovay \cite{Solovay1978Hyperarithmetically} proved that an oracle $D$ is hyperarithmetic if and only if for every infinite set $X\subseteq\omega$,
there exists a $Y\subseteq X$ such that $Y$ computes $D$. Here $g\geq f$ and subset relation are combinatorial notions.
Recently,
Dorais, Dzhafarov, Hirst, Mileti, Shafer \cite{Dorais2016uniform};
Wang \cite{Wang2014Some}; and
Cholak, Patey \cite{cholak2020thin} characterize the set that is encodable 
by a Ramsey type theorem: $\msf{RT}^n_{k,l}$
(where an instance is a coloring $c:[\omega]^n\rightarrow k$ and a solution is an infinite 
set $G\subseteq \omega$ such that $|c([G]^n)|\leq l$). An oracle $D$
is $\msf{P}$-encodable if there is a $\msf{P}$-instance $X $ such that
 every solution of $X$ computes $D$. 
 They showed that 
 the $\msf{RT}^n_{<\infty,l}$-encodable sets are exactly the hyperarithmetic sets
 when $l<2^{n-1}$
 \cite{Dorais2016uniform};
the arithmetic sets when $2^{n-1}\leq l< d_n$ \cite{Dorais2016uniform}
 (where $d_0,d_1,\cdots$ is the sequence of Catalan numbers);
  the computable sets when $l\geq d_n$ \cite{Wang2014Some}.

\item Many computability theoretic researches study the complexity of certain combinatorial notions.
Jockusch \cite{Jockusch1972Ramseys} proved that  there exists a computable coloring $c:[\omega]^3\rightarrow 2$
such that for every infinite  set $G$ so that $|c([G]^3)|=1$, $G$ computes $\emptyset'$.
This  can be read as that some computable coloring $c:[\omega]^3\rightarrow 2$ encode $\emptyset'$.
Kumabee and Lewis \cite{Kumabe2009fixed} proved that there exists a $\msf{DNR}$ $Y$ that has minimal Turing degree
(a Turing degree is minimal if below which there is no nonzero Turing degree).
Recently, Miller and Khan \cite{khan2017forcing} showed that actually, for every $\msf{DNR}$-instance $X$
(which is simply an $X\in\omega^\omega$), there exists a solution $Y$ to $X$
(which is a $Y\in\omega^\omega$ such that $Y(n)\ne X(n)$ for all $n$)
such that $Y$ has minimal Turing degree. Here   ``$\msf{DNR}$-solution to
$X$" is a combinatorial notion.

\item Researches in computability theory use more and more sophisticated combinatorial or probability technique.
For example, Csima, Dzhafarov, Hirschfeldt, Jockusch, Solomon and Westrick \cite{csima2019reverse} proved that
there is a computable coloring $c:\omega\rightarrow 2$ such that for every infinite set $G$
with $G+G=\{x+y:x\ne y\in G\}$ being monochromatic for $c$, $G$ is not computable.
Their proof heavily relies on Lovasz Local Lemma.
Monin, Liu and Patey \cite{liu2019computable} computes, using a particular arithmetic oracle,
an ordered variable word solution of a given computable coloring $c:2^{<\omega}\rightarrow 2$ (which is an $\omega$-variable word $v$ such that the set $\{v(\v a):\v a\in 2^{<\omega}\}$ is monochromatic for $c$
and all occurrence of $x_m$ is before that of $x_{m+1}$, see Definition \ref{combdef2}).
Their construction uses a similar technique as Shelah's proof of Hales-Jewett theorem \cite{shelah1988primitive}.

\end{enumerate}

\ \\

\noindent\textbf{Organization.}
In section \ref{combsec0} we introduce the
computability theoretic question
(Question \ref{combques0})
and present some related results in computability theory
and reverse mathematics.
In section \ref{combsec1}, we introduce the
Ramsian type property (Definition \ref{combdef0}) and show that
  the computability theoretic question
is equivalent to a related combinatorial question
(Theorem \ref{combth0}).
In section \ref{combsec22}, we firstly explain how 
the combinatorial question is related to Hales-Jewett theorem
(Proposition \ref{comblem0} and Remark \ref{combremark0}).
Then we prove, in section \ref{combsec2}, the negation of the combinatorial question in some special case
(Theorem \ref{combth2}),
which turns out to be a generalization of Hales-Jewett theorem (on some parameters).
\ \\

\noindent\textbf{Notation.}
For a sequence $X$ (of any objects),
we write $X\uhr r$ for the sequence
$X(0)\cdots X(r-1)$;
for a set $P = \{r_0<r_1<\cdots\}\subseteq \omega$,
we write $X\uhr P$ for the
sequence $X(r_0)X(r_1)\cdots$;
we write $X\uhr_s^r$ (where $r\geq s$) for the sequence
$X(s)\cdots X(r)$.
For two sequences $\v a,\v b$, we write $\v a^\smallfrown \v b$
(or $\v a\v b$) for the concatenation of
$\v a,\v b$.
For two sets $P_0,P_1\subseteq \omega$,
we write $P_0<P_1$ if $x<y$ for all
$x\in P_0$ and all $y\in P_1$
(note that $\emptyset<P,P<\emptyset$
hold trivially).

\section{Variable word problem}

\label{combsec0}

This section introduces a question of Miller and Solomon.
The question asked that whether every computable
coloring $c:d^{<\omega}\rightarrow k$ admits
a computable variable word infinite solution
(Question \ref{combques0}).
Before giving a concrete definition, we need some notation
on variable words.

\begin{definition}[variable word]
\label{combdef1}
Given $d,k\in \omega$,
\begin{itemize}

\item A \emph{variable word over $d$}
is a sequence $v$ (finite or infinite)
of $\{0,\cdots,d-1\}\cup\{x_0, x_1,\cdots\}$
\footnote{When $d$ is clear, we omit ``over $d$".}.
We say
$x_m$ \emph{occurs} in $v$ if $v(t) = x_m$ for some $t$.

\item
We write $P_{x_m}(v)$ for the set
of coordinates of $x_m$ in $v$, namely $\{t: v(t)  = x_m\}$;
the \emph{first occurrence} of a variable $x_m$ in $v$ refers
to the integer $\min P_{x_m}(v)$.

\item
An \emph{$n$-variable word over $d$}
is a variable word over $d$
such that there are exactly $n$ many  variables
occurring in $v$;
without loss of generality we assume that
the first occurrence of $x_m$ is smaller than
that of $x_{\h m}$ for all $m<\h m$
whenever they both occur in $v$.

\item Given an $\vec{a} \in d^{\h n}$,
an $n$-variable word $v$,
suppose $x_{m_0},x_{m_1},\cdots, x_{m_{n-1}}$ occur in $v$
with $m_{\t n-1}<m_{\t n}$ for all $\t n<n$.
We write $v(\vec{a})$ for the
$\{0,\cdots,d-1\}$-string obtained by substituting
$x_{m_{\t n}}$ with $\v a(\t n)$ in $v$
for all $\t n<\min\{n,\h n\}$ and then truncating the result
just before the first occurrence of $x_{m_{\h n}}$
(when $\h n\geq n$, the first occurrence of $x_{m_{\h n}}$ is set to be $|v|$)
\footnote{We emphasise that $v(\v a)$ is defined even when $\h n> n$.}.

\end{itemize}

\end{definition}

Given an $n$-variable word $v$ over $d$ of length $N$,
the set $\{v(\v a):\v a\in d^n\}$ is called an
$n$-dimensional combinatorial subspace of
$d^N$.
The following is an example of the variable word notation.
\begin{example}
An infinite variable word   over
$\{0,1\}$:
\begin{align}\nonumber
v=&011\ \  x_0x_0 \ 011
\ \  x_1 \  x_0 x_0
\ \   x_1 x_1
\ \ 00\ \    x_2x_2 \cdots
\\ \nonumber
\text{let }\vec{a} =& 10, \text{ then }
\\ \nonumber
v(\vec{a}) =
&011\ \  11 \ \ \ \ 011
\ \ \ 0  \ \ \  11
\ \ \ \  00
\ \ \ \ \ 00;
\\ \nonumber
\PP_{x_0}(v) =& \{3,4,9,10,\cdots\};&
\\ \nonumber
\PP_{x_1}(v) = &\{8,11,12,\cdots\}.&
\end{align}

\end{example}

In \cite{Carlson1984dual}, Carlson  and Simpson introduced the Variable Word   problem, where 
they use it to prove the Dual Ramsey theorem.
Friedman and Simpson~\cite{Friedman2000Issues}, and later Montalban~\cite{Montalban2011Open}, asked about the proof strength of the Variable Word problem.
\begin{definition}[$\msf{VWI},\msf{VW},\msf{OVW}$]
\label{combdef2}
Given $k,d\in\omega$,
\begin{itemize}
\item an instance of 
Variable Word Infinite (henceforth
$\msf{VWI}(d,k)$), Variable Word (henceforth $\msf{VW}(d,k)$) or 
Ordered Variable Word (henceforth $\msf{OVW}(d,k)$)
is a coloring $c:d^{<\omega}\rightarrow k$;

\item a $\msf{VWI}(d,k)$-solution of $c$ is a $\omega$-variable word
$v$ over $d$ such that the set $\{v(\v a):\v a\in 2^{<\omega}\}$
is monochromatic for $c$;

\item moreover, $v$ is a $\msf{VW}$
($\msf{OVW}$ respectively)-solution of $c$
if, in addition: $\PP_{x_m}(v)$ is finite
($\PP_{x_m}(v)<\PP_{x_{m+1}}(v)$
respectively)
for all $m\in\omega$.
\end{itemize}
\end{definition}

We present some computability theoretic results on the
 proof strength of variable word problems.
 We assume that readers are familiar with the basics of
 reverse mathematics such as  Recursion Comprehension
 Axiom (denoted as $\msf{RCA}_0$)
 and Weak K$\ddot{o}$nig's Lemma (denoted as $\msf{WKL}$).
But not knowing these knowledge does not affect reading the main result
(Theorem \ref{combth0}, Theorem \ref{combth2}) of this paper.
 For an introduction to reverse mathematics, see \cite{Simpson2009Subsystems};
 or see \cite{hirschfeldt2015slicing} for a recent development in this area.
Clearly $\msf{OVW}$ has the strongest requirement on its solution
while $\msf{VWI} $ has the weakest requirement on its solution.
Therefore, over $\msf{RCA}_0$,
$\msf{OVW}(d,k)\rightarrow \msf{VW}(d,k)\rightarrow \msf{VWI}(d,k)$.
However, their relation is unknown when the dimension index $d$ are different.
On the other hand, the variable word problems admit,
as many other Ramsian type problems,
an iterative property, i.e.,
a $k$-coloring problem can be solved by repeatedly invoking
a $2$-coloring problem explained as following.
Given a $k$-coloring $c:d^{<\omega}\rightarrow k$,
merge $k-1$ many colors into color $0$ and get a $2$-coloring
$\h c:d^{<\omega}\rightarrow 2$.
Let $v$ be a solution to $\h c$. Suppose
$c(\{v(\v a):\v a\in d^{<\omega}\})\subseteq \{0,\cdots,k-1\}$
(otherwise we are done).
Now $v$ together with $c$ give rise to a $(k-1)$-coloring $\t c$ on
$d^{<\omega}$ defined as following:
$$\t c(\v a) = c(v(\v a))\text{ for all }\v a\in d^{<\omega}.$$
Clearly, given a solution $\t v$ to $\t c$, we can compute
a solution to $c$ using $\t v$ and $v$.
Thus, we have, over $\msf{RCA}_0$, $\msf{V}(d,k)\leftrightarrow \msf{V}(d,k+1)$
for all $k\geq 2$ where
$\msf{V}$ is $\msf{OVW},\msf{VW}$ or $\msf{VWI}$.

In \cite{Miller2004Effectiveness}, Miller and Solomon proved that there is a
computable $\msf{OVW}(2,2)$-instance that does not admit a
$\Delta_2^0$ solution.
This implies that over $\msf{RCA}_0$,
 $\msf{WKL}$ does not prove $\msf{VW}(2,2)$.
 To see this, note that for every $\msf{VW}(2,2)$-solution $v$ of $c$,
 $v'$ computes a $\msf{OVW}(2,2)$-solution of $c$ in the way that
 $v'$ could compute a sequence $m_0<m_1<\cdots$
 of integers such that $P_{x_{m_r}}(v)<P_{x_{m_{r+1}}}(v)$ for all $r\in\omega$.
 On the other hand, if $\msf{WKL}$ proves $\msf{VW}(2,2)$, then every
 computable $\msf{VW}(2,2)$-instance admits a low solution, thus every
 computable $\msf{OVW}(2,2)$-instance admits a $\Delta_2^0$ solution, a contradiction.

In \cite{liu2019computable}, Liu, Monin and Patey showed that $\msf{ACA}$
proves $\msf{OVW}(2,2)$. Actually, they showed that
for every computable $\msf{OVW}(2,2)$-instance $c$,
 every $\emptyset'$-PA degree computes a solution of $c$.
 On the other hand, they also construct a computable $\msf{OVW}(2,2)$-instance
 such that every solution is of $\emptyset'$-DNC degree.
The variable word problems are closely related to Hindman's theorem and the Dual Ramsey theorem
(see e.g. \cite{dzhafarov2017effectiveness}).

A lot of questions on variable word problems are unknown,
say the proof strength lower bound, the relations between the three
versions of the variable word problem, the proof strength upper bound of
$\msf{VWI}$ etc.
Miller and Solomon asked the following questions
\begin{question}[Miller and Solomon]\label{combques0}
Does $\msf{RCA}_0$ prove $\msf{VWI}(2,2)$?
Or, in terms of computability theoretic language,
does every   $\msf{VWI}(2,2)$-instance $c$ admit a $c$-computable solution?

\end{question}
We will show that   Question \ref{combques0}
is equivalent to a natural combinatorial question.
Concerning the lower bound of these problems,
\begin{question}[Miller and Solomon]
\label{combques1}
Is there a $d\in\omega$
such that $\msf{OVW}(d,2) $ proves
$\msf{ACA}$?
On the other hand, it is not even known that whether
$\msf{VWI}(2,2)$ proves $\msf{ACA}$.

\end{question}
Usually, Question \ref{combques1} is closely related to the question
whether a computable $\msf{OVW}(d,2)$-instance could encode an incomputable oracle,
i.e., whether there is an incomputable oracle $D$ and a computable $\msf{OVW}(d,2)$-instance
such that every solution computes $D$.
However, we don't even know if an arbitrary instance (not necessarily computable)
of $\msf{OVW}(d,2)$ could encode an incomputable oracle.
\begin{question}
Is there an incomputable oracle $D$, a $d\in\omega$ and
a $\msf{OVW}(d,2)$-instance $c$ (not necessarily computable)
such that every solution of $c$ computes $D$.
\end{question}
We also wonder the relation between
the three versions of the variable word problem.
\begin{question}
Given a $d\in\omega$,
 over $\msf{RCA}_0$,
does $\msf{VW}(d,2)$ prove $\msf{OVW}(2,2)$;
does
 $\msf{VWI}(d,2)$ prove $\msf{VW}(2,2)$?
\end{question}
We also wonder whether the dimension affects the proof strength of variable word problem.
\begin{question}
Given a $d\in\omega$, over $\msf{RCA}_0$,
does $\msf{OVW}(d,2)$
($\msf{VW}(d,2), \msf{VWI}(d,2)$ respectively) prove
$\msf{OVW}(d+1,2)$
($\msf{VW}(d+1,2), \msf{VWI}(d+1,2)$ respectively)?

\end{question}
\section{The combinatorial equivalence}
\label{combsec1}

In this section, we show that
 Question \ref{combques0} is equivalent to
a combinatorial question (see Theorem \ref{combth0}). The combinatorial
question
asked whether there is an infinite sequence of integers
such that every initial segment of the sequence satisfies
a Ramsian type combinatorial property defined in Definition \ref{combdef0}.

\begin{definition}[$\opp_k^d$]\label{combdef0}
Let $n_0,n_1,\cdots,n_{r-1}$ be a sequence
of positive integers, let $N_0 = \{0,\cdots, n_0-1\}$,
$N_1 = \{n_0,\cdots, n_0+n_1-1\}$,
$\cdots$,
$N_{r-1}=\{n_0+\cdots+n_{r-2},\cdots, n_0+\cdots+n_{r-1}-1\}$,
and $N  = \cup_{s\leq r-1}N_s$;
let $f:d^N\rightarrow k$.
We say $n_0\cdots n_r$ is \emph{\sectionhomo} for $f$
if there exists an $s\leq r-1$,
an $n_s$-variable word $v$ over $d$ of length $N$
such that 
the first occurrence of variables in $v$ consist
of $N_s$, i.e.,
$$\{\min \PP_{x_m}(v):m\in\omega\}
=N_s,$$
and $v$ is monochromatic for $f$.

We write $\opp^d_k(n_0\cdots n_{r-1})$
iff there exists a coloring $f:d^N\rightarrow k$
such that 
$n_0\cdots n_{r-1}$ is \emph{not} \sectionhomo\ for $f$
\footnote{\opp\ is short for: existence of a coloring $f$ such that the sequence is non-\sectionhomo\ for $f$.}.
In that case we say $f$ witnesses
$\opp^d_k(n_0\cdots n_{r-1})$.

\end{definition}

\def\mb{\mathbf}
We give some intuition of $\opp_k^d$
by giving some simple observations and examples.
For $\v{n},\v{\h n}\in\omega^{<\omega}$
we write $\v{n}\leq \v{\h n}$ if $|\v{n}| = |\v{\h n}|$
and
$\v{n}(s)\leq \v{\h n}(s)$ for all $s<|\v{n}|$.
We say $\v{n}$ is a subsequence of $\v{\h n}$
if there are integers $s_0<s_1<\cdots<s_{m-1}<|\v{\h n}|$
such that $\v{n} = \v{\h n}(s_0)\cdots\v{\h n}(s_{m-1})$.
It's obvious that:
\begin{proposition}\label{combprop0}
If $\v{n}$ is a subsequence of $\v{\h n}$
or $\v{n}\geq \v{\h n}$,
then
$\opp_k^d(\v{\h n})$ implies
$\opp_k^d(\v{n})$.

\end{proposition}
\begin{proof}
For $\v{n}$ being a subsequence of $\v{\h n}$,
suppose $\v{n} = \v{\h n}(s_0)\cdots\v{\h n}(s_{r-1})$
and $\h f$ witnesses $\opp_k^d(\v{\h n})$.
We will construct a witness for $\opp_k^d(\v{n})$ according to $f$.
Let $N = \v{n}(0)+\cdots + \v{n}(r-1)$,
$\h N = \v{\h n}(0)+\cdots+\v{\h n}(|\v{\h n}|-1)$.
The function $f$ is constructed by embedding $d^N$ in to
$d^{\h N}$ (see (\ref{combdef0})).
For every $\v a \in d^N$, suppose $\v a = \v a_0\v a_1\cdots\v a_{r-1}$
where $|\v a_s| = \v{n}(s)$ for all $s<r$,
let
\begin{align}\label{combeq0}
\v{\h a} = \underbrace{0\cdots 0}
\limits_{\v{\h n}(0)+\cdots +\v{\h n}(s_0-1)} \v a_0 \underbrace{0\cdots 0}
\limits_{\v{\h n}(s_0+1)+\cdots+ \v{\h n}(s_1-1)} \v a_1
\underbrace{0\cdots0}
\limits_{\v{\h n}(s_1+1)+\cdots +\v{\h n}(s_2-1)} \v a_2\cdots
\end{align}
and set
$$
f(\v a) = \h f(\v{\h a}).
$$

To show that $f $ witnesses
$\opp_k^d(\v{n})$, let $s<r$ and
let $v$ be an $\v{n}(s)$-variable word
over $d$
of length $N$
with $\{\min\PP_{x_{m}}(v):m\in\omega\}=N_s$.
Let $\h v$ be the $\v{n}(s)$-variable
word over $d$ of length $\h N = \v{\h n}(0)+\cdots+\v{\h n}(|\v{\h n}|-1)$
defined in the same fashion as (\ref{combeq0}).
Since $\h f$ witnesses $\opp_k^d(\v{\h n})$,
there exists $\v b_0,\v b_1\in d^{\v{n}(s)}$
such that $\h f(\h v(\v b_0))\ne \h f(\h v(\v b_1))$.
Clearly, we have $f(v(\v b_0))\ne f(v(\v b_1))$,
thus we are done.

For $\v{n}=n_0\cdots n_{r-1}\geq \v{\h n} = \h n_0\cdots \h n_{r-1}$,
suppose $\h f$ witness $\opp_k^d(\v{\h n})$.
Note that we can project $d^{N}$ into $d^{\h N}$ as following.
For every $\v a = \v a_0\cdots \v a_{r-1}\in d^N$ where $|\v a_s | = \v{n}(s)$ for all $s<r$, let
$\v{\h a}_s = \v a_s\uhr \h n_s $,
and let $\v{\h a} = \v{\h a}_0\cdots \v{\h a}_{r-1}$.
Set $f(\v a) = \h f(\v{\h a})$ for all $\v a\in d^N$.
It's easy to verify that $f$ witnesses $\opp_k^d(\v{n})$.

\end{proof}

\begin{proposition}
If $\h d\geq d,\h k\geq k$, then
$\opp_k^d(\v{ n})$ implies
$\opp_{\h k}^{\h d}(\v{n})$.

\end{proposition}
\begin{proof}
Obviously, we can embed $d^N$ into $ \h d^{ N}$.
Thus the proof follows in a similar fashion as Proposition \ref{combprop0}.
\end{proof}

\begin{proposition}\label{combprop1}
We have $\opp_2^2(22),
\opp_2^2(222)$
and
$\opp_2^2(n)$
for all $n>0$.

\end{proposition}
\begin{proof}
To see $\opp_2^2(22)$,
consider $$f(\v a) =
\v a(0)+\v a(1)+\v a(2)
\ mod\ 2.$$
To see $\opp_2^2(222)$,
consider
$$f(\v a) =
I(\v a(0)+\v a(1)>0)+
\v a(2)+\v a(3)
+\v a(4)\ mod\ 2.$$
Where $I$
 is the indication function.
 To see $\opp_2^2(n)$,
simply consider
$f(\v a) =
\v a(0)\ mod\ 2.$
\end{proof}

Now comes our main result.
Let $\opp_k^d$ denote the set of
$\omega$-sequences $X$ of positive integers
such that for every $r>0$,
$\opp_k^d(X\uhr r)$.
\begin{theorem}\label{combth0}
The followings are equivalent:
\begin{enumerate}
\item There exists an $X\in\opp_k^d$.

\item There exists a $\msf{VWI}(d,k)$-instance
$c$ that does not admit a $c$-computable solution.
\end{enumerate}

\end{theorem}

We emphasis that in item (1) of Theorem \ref{combth0},
there is no complexity restriction on $X$.
Therefore item (1) is a natural combinatorial assertion.
It's also clear that item (2) of Theorem \ref{combth0}
is a negation of   Question \ref{combques0}.
We say a Turing functional on an oracle $X$, namely $\Psi^X$, computes a  variable word
if $\Psi^X$ is an $X$-c.e. set
(possibly finite)
$\{v_0,v_1,\cdots\}$ of finitely long variable words
such that $v_0\preceq v_1\preceq\cdots$;
we write $\Psi^X[t]$ to denote this enumerable set computed by time $t$;
for convenience, we assume that $x_{m+1}$ occurs in $v_n$ implies $x_m$ occurs in $v_n$
for all $m,n\in\omega$.
We say $\Psi^X$ is total if $\cup_{v\in\Psi^X}v$ is an $\omega$-variable word;
$\Psi^X$ is a $\msf{VWI}(d,k)$-solution to $c:d^{<\omega}\rightarrow k$ if it is total and
for every $v\in \Psi^X$, $|c(\{v(\v a):\v a\in d^{<\omega}\})|=1$.

\begin{proof}

(1)$\Rightarrow$(2):
We will use $X$ to compute a coloring $c:d^{<\omega}\rightarrow k$
such that $c$ does not admit $X$-computable $\msf{VWI}(d,k)$-solution.
We need $c$ to diagonal against $\Psi_0^X,\Psi_1^X,\cdots$
where each Turing functional is computing a variable word.

To illustrate the combinatorial idea of the proof
(and put priority argument aside), we temporarily assume that
each Turing functional is total.
Therefore,
\begin{align}\label{combeq3}
&\text{for each $r\in\omega$,
let $v_r\in\Psi_r^X$ be such that
$v_r$ contains $X(r)$ many}\\ \nonumber
&\text{variables whose first occurrence is after $|v_{r-1}|$
.}
\end{align}
\footnote{We set $|v_{-1}|=0$.}
Without loss of generality, assume that
all variables in $v_r$ occur after $|v_{r-1}|$
(otherwise substitute
the exceptional  variables by $0$); moreover, suppose
  the variables occurring in $v_r$
are $\{x_m: m<X(r)\}$.
Suppose $f_r$ witnesses $\opp_k^d(X\uhr r)$,
we will transform these $(f_r: r>0)$ into a coloring $c$
such that none of these $v_r$ can be extended to a solution in the sense
that
\begin{align}\label{combeq1}
&\text{there is no   $X(r)$-variable word $v\succeq v_r$
}\text{such that $ v$ is monochromatic for
$c$,}\\ \nonumber
&\text{ namely $|c(\{v(\v a): \v a\in d^{X(r)}\})|=1$.
 }
 \end{align}

 For each $r\in\omega$, let $P_r$
be the set $\{\min P_{x_m}(v_r): m<X(r) \}$
of the first occurrence of the variables in $v_r$.
Clearly $|P_r| = X(r)$ and $P_r<P_{r+1}$ for all $r\in\omega$.
For every $n\in\omega$,
let $r(n)$ be the maximal integer such that
$|v_{r(n)}|\leq n$, we define
$c$ on $d^n$ so that no $v_r$ with $r\leq r(n)$ can be extended
to a $ v$ with $|v|=n$ that is monochromatic for $c$.
If $r(n)$ is undefined, then define $c$ on $d^n$ arbitrarily.
 Otherwise, we define $c$ on $d^n$ so that
$c(\v a)$ depends on $\v a\uhr \cup_{r\leq r(n)}P_r$.
More specifically, for $\v a\in d^n$, let
\begin{align}\label{combeq2}
c(\v a) = f_{r(n)+1}(\v a\uhr \cup_{r\leq r(n)}P_r).
\end{align}

 We now verify (\ref{combeq1}).
 Fix an $r\in\omega$,
 an $X(r)$-variable word $v\succeq v_r$.
 Suppose $|v| = n$.
 It suffices to show that there are $\v b_0,\v b_1\in d^{X(r)}$
 such that $c(v(\v b_0))\ne c( v(\v b_1))$.

Let $j$ be the 1-1 order preserving map
from $P=\cup_{r\in\omega} P_r$ to $\omega$.
Let $$\h v =  v\uhr \cup_{r\leq r(n)}P_r$$
and note that $\h v$
 is an $X(r)$-variable word of length
 $X(0)+\cdots+ X(r(n))$
 (variables occurring in $\h v$ are $\{x_m:m<X(r) \}$)
 such that, for every $m<X(r) $,
\begin{align}\nonumber
 \min\PP_{x_m}(\h v)& =\min j(\PP_{x_m}(v)\cap (\cup_{r\leq r(n)}P_r))
  \\ \nonumber
  &=j(\min \PP_{x_m}(v))= X(0)+\cdots + X(r-1)+m.
  \end{align}
Thus,
\begin{align}\nonumber
\{\min  \PP_{x_m}(\h v):m<X(r)\}
= N_r
\end{align}
where $N_r = \{X(0)+\cdots +X(r-1),\cdots,
X(0)+\cdots+ X(r)-1\}$.
By definition of $f_{r(n)+1}$,
there exist $\v b_0,\v b_1\in d^{X(r)}$
such that
$$
f_{r(n)+1}(\h v(\v b_0))
\ne f_{r(n)+1}(\h v(\v b_1)).
$$

But $\h v(\v b) = v(\v b)\uhr \cup_{r\leq r(n)}P_r$
for all $\v b\in d^{X(r)}$.
Thus, combine with the definition of $c$ (see (\ref{combeq2})),
we have
\begin{align}\nonumber
&c( v(\v b_0))
\\ \nonumber
=&
f_{r(n)+1}( v(\v b_0)\uhr \cup_{r\leq r(n)}P_r)
 \\ \nonumber
 =& f_{r(n)+1}(\h v(\v b_0))
 \\ \nonumber
 \ne&
 f_{r(n)+1}(\h v(\v b_1))
 \\ \nonumber
 =& f_{r(n)+1}( v(\v b_1)\uhr \cup_{r\leq r(n)}P_r)
  \\ \nonumber
  =& c( v(\v b_1)).
\end{align}
Thus we are done.

For the general case where (\ref{combeq3})
is not necessarily true, we maintain,
at each time $t$ a finite set $W[t] = \{w_0<w_1<\cdots\}\subseteq \{0,\cdots, t-1\}$
such that (\ref{combeq3}) is true
for all $r< |W[t]|$. i.e.,
\begin{align}\label{combeq4}
&\text{for each $r< |W[t]|$,
there exists a $v_r\in\Psi_{w_r}^X[t]$ such that
$v_r$ contains $X(r) $ many}\\ \nonumber
&\text{variables whose first occurrence is after $|v_{r-1}|$.}
\end{align}
Moreover, we assume that
\begin{align}\label{combeq5}
\text{
for every $r<| W[t]|$,
$|v_r|<t$.}
\end{align}
If $W[t]=\emptyset$, define $c$ on $d^t$ arbitrarily.
Otherwise, define
$c$ on $d^t$ as
$$c(\v a) = f_{|W[t]|}(\v a\uhr \cup_{r< |W[t]|}P_r).$$

Note that as we argued before, if for some $w, t\in\omega$,
$W[\h t]\cap [0,w]  = W[t]\cap [0,w]$ for all $\h t\geq t$,
suppose $W[t]\cap [0,w] = \{w_0<w_1<\cdots<w_r<\cdots\}$,
we have that $v_r\in \Psi_{w_r}^X$ can not be extended
to a solution in a similar  sense of (\ref{combeq1}). i.e.,
\begin{align}\nonumber
&\text{there is no $X(r)$-variable word $v\succeq v_r$ with $|v|\geq t$  }
\\ \nonumber
&\text{that is monochromatic for $c$.}
\end{align}
Thus $c$ must diagonal against $\Psi^X_{w_r}$ if $\Psi^X_{w_r}$ is total.
By a simple priority argument, we can $X$-compute
a  sequence $(W[t]:t\in\omega)$ of finite sets such that
\begin{itemize}
\item property
(\ref{combeq4}) and (\ref{combeq5}) follow;
\item  $W[t]\cap [0,w]$ converges
as $t$ tends to infinity for all $w\in\omega$;
\item let $W$ be the limit of $W[t]$,
we have $W$ is maximal in the sense that
$\Psi_w^X$ is not total for all $w\notin W$.
\end{itemize}
Thus this means, for every $w$, either $c$ diagonal againsts $\Psi_w^X$
or $\Psi_w^X$ is not total.

At last we notice that the proof actually shows that for
an oracle $D$  such that $D'$ computes
some $X\in \opp_k^d$, there is a $D$-computable coloring
$c:d^{<\omega}\rightarrow k$ such that $c$ does not admit
$D$-computable $\msf{VWI}(d,k)$-solution
(see Theorem \ref{combth1}).
To see this, we $D$-compute
a  sequence $(W[t]:t\in\omega) $ of finite sets such that
for every $t$, let $W[t] = \{w_0<w_1<\cdots\}$,
we have
\begin{align}\nonumber
&\text{for each $r< |W[t]|$,
there exists a $v_r\in\Psi_{w_r}^{D}[t]$ such that
$v_r$ contains }\\ \nonumber
&\text{$X[t](r)$ many
variables whose first occurrence is after $|v_{r-1}|$.}
\end{align}

\ \\

(2)$\Rightarrow$(1):
Fix a $c:d^{<\omega}\rightarrow k$ that does not admit a
$c$-computable $\msf{VWI}(d,k)$-solution.
We will take advantage of some particular algorithms
$\Phi_0^c,\Phi_1^c,\cdots$ and show that
their failure (to compute a solution to $c$)
give rise to a sequence $X\in\opp_k^d$.
Roughly speaking,
the algorithms $\Phi_0^c,\Phi_1^c,\cdots$
 are greedy algorithms in the sense
 that
\begin{align}\nonumber
&\text{they extend their current computation
(which is a finitely long variable word)}\\ \nonumber
 &\text{whenever
  possible. i.e., }
  \\ \nonumber
  &\hspace{0.5cm}\Phi_{r+1}^c \text{ extends its current computation from }
  v
  \text{ to some }
\h v\succeq v
\\ \nonumber
&\hspace{0.5cm}\text{where } \h v\text{ has more variables than }v,
\\ \nonumber
&\hspace{0.5cm}\text{whenever it is found that }
  \text{for \emph{some} $\v a\in d^{|v_r|+1}$,
  $\h v/\v a$ is monochromatic for $c$}\\ \nonumber
  &\hspace{0.5cm}\text{where }v_r\text{ is the current computation of }\Phi_r^c.
  \end{align}
 Moreover, $\Phi_{r+1}^c$ will build its solution $v_{r+1}^*$ based on
  $\Phi_0^c,\cdots,\Phi_r^c$ in the sense that
  \begin{align}\nonumber
  &\text{all variables
  in $v_{r+1}^*$ occur after $|v_r|$ and}
  \\ \nonumber
  &\text{if some }\Phi_{\t r}^c
  \text{ extends its current computation,}
  \\ \nonumber
  &\text{then all }\Phi_r^c   \ (\text{where }r>\t r)
  \text{ will restart all over again}.
  \end{align}

Since $c$ does not admit a $c$-computable solution,
for every $r\in\omega$, the computation of $\Phi_r^c$
stucks at some $v_r$. More concretely,
\begin{itemize}
\item There is no  $\h v\succeq v_r$
with $\h v$ containing more variables than $v_r$
such that for some $\v a\in d^{|v_{r-1}|+1}$,
$\h v/\v a$ is monochromatic for $c$.
 \item All variables in $v_{r}$ occur after
 $|v_{r-1}|$
 (meaning $P_{x_m}(v_r)>|v_{r-1}|$ for all $m\in\omega$)
 and $|v_r|>|v_{r-1}|$.

\end{itemize}
Let $\h v_r$ be $\v v_r^\smallfrown x_{n_r-1}$,
where we assume that the variables occurring in $v_r$
are $\{x_0,\cdots,x_{n_r-2}\}$.
Then the first item transforms into:
\begin{align}\label{combeq6}
&\text{there is no }n_r\text{-variable word }
\h v\succeq \h v_r
\text{ such that
for some $\v a\in d^{|\h v_{r-1}|}$,}\\ \nonumber
&\h v/\v a\text{ is monochromatic for }c.
\end{align}
We show that $n_0n_1n_2\cdots\in\opp_k^d$.
Fix an $r\in\omega$, we need to construct an
$f:d^N\rightarrow k$ witnessing $\opp_k^d(n_0\cdots n_r)$
where $N = n_0+\cdots + n_r$.
For every $\v a\in d^N$, to define $f(\v a)$,
\begin{align}\nonumber
\text{we map $\v a $ to a word $\v{\h a} = h(\v a)\in d^{|\h v_r|}$
and let $f(\v a) = c(\v{\h a})$.}
\end{align}
Intuitively, $h$ is defined by connecting each element of $N$
to a set of elements of $|\h v_r|$,
namely  $\PP_{x_m}(\h v_s)$ for some
$m,s\in \omega$.
Suppose $\v  a = \v a_0\cdots \v a_r$ where $|\v a_s| = n_s$
for all $s\leq r$.
Let
\begin{align}\nonumber
&\v{\h a}_s = \h v_s(\v a_s)\uhr_{|\h v_{s-1}|}^{|\h v_s|-1}
\text{ and }
h(\v a) = \v{\h a}_0\cdots\v{\h a}_r.
\end{align}
Let $s\leq r$ and $\t v$ be a $n_s$-variable word
such that $\{\min \PP_{x_m}(\t v):m\in\omega\}
 = N_s$ where
 $N_s = \{X(0)+\cdots +X(s-1),\cdots,
 X(0)+\cdots +X(s)-1\}$.
 Without loss of generality, we assume that the
 variables occurring in $\t v$ are the same as that of
 $\h v_s$; moreover, the order of  their occurrence is also identical.
Note that $h$ naturally extends to a map
defined on any sequence of length $N$.
Consider the variable word
$\h v= h(\t v)$
and note that $\h v\uhr_{|\h v_{s-1}|}^{|\h v_s|-1} =
\h v_s\uhr_{|\h v_{s-1}|}^{|\h v_s|-1}$.
Therefore, since  variables in $\h v_{\h s}$ occurs
after $|\h v_s|-1$ for all $\h s>s$,
let $\v a = \h v_s\uhr |\h v_{s-1}|$,
we have
\begin{align}\nonumber
\h v/\v a\text{
is an }n_s\text{-variable word extending }
\h v_s.
\end{align}
By (\ref{combeq6}),
there are $\v b_0,\v b_1\in d^{n_s}$
such that $c(\h v(\v b_0))\ne c(\h v(\v b_1))$.
Since $h(\t v(\v b)) = \h v(\v b)$, we have
\begin{align}
f(\t v(\v b_0))
& = c(h(\t v(\v b_0)))
\\ \nonumber
&  = c(\h v(\v b_0))
\\ \nonumber
&\ne
c(\h v(\v b_1))\\ \nonumber
& = c(h(\t v(\v b_1)))
\\ \nonumber
&= f(\t v(\v b_1)).
\end{align}
Thus we are done.

It worth mention that $c$
does not necessarily compute
$n_0n_1\cdots$ since $c$ can not
decide whether $\Phi_r^c$ will ever
extend its current computation.
However, we have that $c'$ computes
$n_0n_1\cdots$
(see Theorem \ref{combth1}).

\end{proof}

As we mentioned in the proof of Theorem \ref{combth0},
we actually have the following.
\begin{theorem}\label{combth1}
The following classes of oracles are identical:
\begin{itemize}
\item $\big\{
D\subseteq \omega:D' \text{ computes a
member in }\opp_k^d
\big\}.$

\item $\big\{
D\subseteq \omega:D \text{ computes a
 }\msf{VWI}(d,k)\text{-instance }c$
  that does not admit
$c$-computable solution
$\big\}.$
\end{itemize}

\end{theorem}

\begin{remark}
In Theorem \ref{combth0}, direction (2)$\Rightarrow$(1)
is particularly interesting.
It is an example that a computability theoretic
assertion has some corollary outside computability theory.

On the other hand, experience in computability theory
gives us some clue about the complexity of the members in
$\opp_k^d$ (without analyzing $\opp_k^d$).
It suggests that if $\opp_k^d$ is nonempty, then $\opp_k^d$  admits
a $\emptyset'$-computable member. This is because
almost all computability theoretic assertions relativize.
Therefore, if there is a $\msf{VWI}(d,k)$-instance
$c$ that does not admit a $c$-computable solution,
then it is likely that such $c$  can be chosen to be
computable. Thus, by Theorem \ref{combth1},
if $\opp_k^d$ is nonempty, then  it is likely that
$\opp_k^d$ admits
a $\emptyset'$-computable member.

\end{remark}

\section{On the $\opp_k^d$ property}

\label{combsec22}

It turns out that disproving $\opp_k^d$ on
certain sequences  is equivalent
to Hales-Jewett theorem (see Remark \ref{combremark0}).
For $d,k,n\in\omega$,
let $HJ(d,k,n)$ denote the assertion
that
\begin{align}\label{combeq7}
&\text{there exists an }N\text{ such that for every }
c:d^N\rightarrow k, \\ \nonumber
&
\text{there exists an }n\text{-variable word }v
\text{   of length }N\text{ that is monochromatic for  }c.
\end{align}
 Hales-Jewett theorem states that
\begin{theorem}[Hales-Jewett theorem]
For every $d,k,n\in\omega$, $HJ(d,k,n)$ holds.
\end{theorem}
See \cite{dodos2016ramsey} for a proof of Hales-Jewett theorem.
Using Hales-Jewett theorem, we prove the following
Proposition \ref{comblem0}, which also directly  implies
Hales-Jewett theorem.
When there is no ambiguity, 
we use $\nr$ to denote $n\cdots n$ ($r$ many).
\begin{proposition}\label{comblem0}
For every $d,k,n\in\omega$, there exists an $r$
such that $\opp^d_k(\nr)$ does not hold.
\end{proposition}

\begin{proof}
We prove the Lemma for $d= 2,n=2$. The general case follows similarly.
By $HJ(4,k,1)$, let $r$ be a witness (as $N$ in (\ref{combeq7})).
We show that $\opp_k^2(\2r)$ does not hold.
This is done by coding the space $2^{2r}$ into $4^r$ in the obvious way
where $\v a(2 t)\v a(2t+1)$ is coded into $\v{\h a}(t)$.
i.e., for every $\v a\in 2^{2r}$, let $\v{\h a} = h(\v a)\in 4^r$
be such that $\v{\h a}(t) = 0,1,2,3$ respectively
depending on   $\v a(2t)\v a(2t+1)=00,01,10,11$ respectively.
Given a coloring $c:2^{2r}\rightarrow k$, consider
$$\h c:4^r\ni \v{\h a}\mapsto c(h^{-1}(\v{\h a}))\in k$$
and suppose (by  definition of $r$) $\h v$ is a $1$-variable word (over $4$) of length $r$
that is monochromatic for $\h c$.  Without loss of generality, suppose $x_0$
occurs in $\h v$.
We transform $\h v$ back to a $2$-variable word (over $2$) $v$
 of length $2r$ such that
$$v(2t)v(2t+1) = 00,01,10,11,x_0x_1
\text{ respectively if }\h v(t) = 0,1,2,3,x_0
\text{ respectively}.
$$

Note that for every $\v a\in 2^2$,
there exists an $\v{\h a}\in 4$ such that
$ \h v(\v{\h a})  =h( v(\v a))$.
Thus, $v$ is monochromatic for $c$.
Moreover, it's clear
that $\{\min \PP_{x_m}(v):m\in\omega\}
 = \{2t^*,2t^*+1\}$ where $ t^* = \min \PP_{x_0}(\h v)$.
 Thus we are done.

\end{proof}
\begin{remark}\label{combremark0}
Note that
given $d,k,n$, the existence of $r$ (as Lemma \ref{comblem0}) clearly implies $HJ(d,k,n)$
where the witness $N$ could be $n\cdot r$. i.e.,
the assertion $n^\omega\notin \opp_k^d$ for all $d,k,n$ is equivalent to Hales-Jewett theorem.
Thus the assertion $\opp_k^d=\emptyset$ for all $d,k$ is a generalization of Hales-Jewett theorem.
\end{remark}

One possible way to show that $\opp_k^d\ne\emptyset$
is by proving that if $n_0\cdots n_{r-1}\in\opp_k^d$, then
there exists an $n$ such that $n_0\cdots n_{r-1} n\in \opp_k^d$.
However, this clearly doesn't work.
\begin{proposition}
For every $n$,
$\opp_2^2(1n)$ does not hold.

\end{proposition}
\begin{proof}
Fix an $f: 2^{n+1}\rightarrow 2$.
Without loss of generality, assume that
$f(0\cdots 0)  = 0$.
Suppose there is a  $\v a\in 2^{n+1}$
with $\v a(0) = 0$
such that $f(\v a) = 1$ (otherwise the variable word
$v = 0x_0\cdots x_{n-1}$ is monochromatic for
$f$ in color $0$).
If $f(1\cdots 1) = 0$,
then the variable word $v = x_0\cdots x_0$
is monochromatic for $f$ in color $0$.
Otherwise, let $v$ be the $1$-variable word
such that $v(0) = \v a, v(1) = 1\cdots 1$,
we have $v$ is monochromatic for $f$.
Thus we are done.

\end{proof}

We refer to the  argument where we show that there is a
$\v a$ with $\v a(0) = 0$ so that $f(\v a) = 1$ as \emph{\search\ argument}.
More generally, by \search\ argument, we have:
\begin{lemma}\label{comblem1}
Let $n_0\cdots n_r$ be a sufficiently long positive integer sequence
that is not \sectionhomo\ for $f$.
Then there exist
$(\v a_m\in 2^N: m<n_0)$
such that $ \v a_m^{-1}(1)$ are mutually disjoint and
$f(\v a_m) = 0, m\in \v a^{-1}_m(1)$ for all $m<n_0$.

\end{lemma}

For Lemma \ref{comblem1}, to be sufficiently long, it suffices that
$r\geq n_0$. However, in the following subsection, to be sufficiently long
could be very complex. In some cases, we reduce a stronger
statement to a weaker one, and to be sufficiently long in that stronger
statement depends on the weaker statement. For example,
 to be sufficiently long in  that
stronger statement means there are sufficiently many mutually disjoint subsequences of
$n_0\cdots n_r$ so that each is sufficiently long in the sense of that weaker statement.
Since where these subsequences locate do not matter,
therefore, to avoid unnecessary complication and an ocean of notation,
we use the term sufficiently long.
In this term, $\opp_k^d=\emptyset$ can be stated as for every sufficiently
long $n_0\cdots n_r$, every $f:d^N\rightarrow k$,
$n_0\cdots n_r$ is \sectionhomo\ for $f$.

\subsection{A generalization of Hales-Jewett theorem on $d=2,k=2,n=2$.}
\label{combsec2}
Our goal of this subsection is to prove the following generalization
of Hales-Jewett theorem on the parameters $d=2$.
\def\02{[0,2]}
\begin{theorem}\label{combth2}
For every $\omega$-sequence of positive integers
$n_0n_1\cdots$ with $n_0=2$,
$n_0n_1\cdots\notin \opp_2^2$.
\end{theorem}
As explained in Remark
\ref{combremark0},
this generalizes $HJ(2,2,2)$.
For convenience, we prove the
following equivalent version of Theorem \ref{combth2}.
In this section,
a $[0,n]$-variable word $v$ is a variable word
such that the variables occurring in $v$
is a subset of $\{x_0,\cdots,x_{n-1}\}$.
Abusing the notation,
we write $v(\v a)$ for the string obtained by substituting $x_m$
by $\v a(m)$ for all $m<n$.
For example, let variable word $v = 0x_10$ and $\v a=01$,
then $v(\v a) = 010$
(although Definition \ref{combdef1} gives $v(\v a) = 000$).

\begin{theorem}
\label{combth3}
Let
$n_0\cdots n_r$ be sufficiently long
and colorings $f_{00},f_{01},f_{10},f_{11}: 2^N\rightarrow 2$.
Then,
either
$n_0\cdots n_r$ is  \sectionhomo\ for
  $f_{\v b}$ for some $\v b\in 2^2$;
or there exists a $\02$-variable word $v$
such that $f_{\v b}(v(\v b)) = f_{\v{\h b}}(v(\v{\h b}))$
for all $\v b,\v{\h b}\in 2^2$.

\end{theorem}
\begin{proof}[Proof of Theorem \ref{combth2} from Theorem \ref{combth3}]
Fix an $f:2^N\rightarrow 2$.
Let
$$f_{\v b}:2^{N-2}\ni \v a\mapsto
f(\v b \v a). $$
If $n_1\cdots n_r$ is \sectionhomo\ for $f_{\v b}$
for some $\v b$, then clearly we are done
(since it means $n_0\cdots n_r$ is \sectionhomo\ for $f$).
Otherwise, by Theorem \ref{combth3},
let $v$ be a $\02$-variable word of length $N-2$
such that  for some $j\in 2$,
$f_{\v b}(v(\v b)) = j$ for all $\v b\in 2^2$.
Let $\h v = x_0x_1 v$.
Clearly $f(\h v(\v b)) = f(\v b\cdot v(\v b)) = f_{\v b}(v(\v b)) = j$.
Thus $f$ does not witnesses $\opp_2^2(n_0\cdots n_r)$.
The conclusion follows since $f$ is arbitrary.
\end{proof}
Of course, Theorem \ref{combth2} implies Theorem \ref{combth3} in the same fashion.
The remaining of section \ref{combsec2} will prove Theorem \ref{combth3}.
For every $\v a\in 2^{<\omega}$, we think
of $\v a $ as a  finite set $\v a^{-1}(1)$
(and vice versa when the length of $\v a$ is clear),
therefore it makes sense to write $\v a_0\cup \v a_1, \v a_0\cap \v a_1,\v a_0\setminus\v a_1$
\footnote{We only use this sort of notation when $|\v a_0|=|\v a_1|$.}.

Assume otherwise that
$n_0\cdots n_r$ is not \sectionhomo\ for $f_{\v b}$ for all $\v b\in 2^2$
and
the $\02$-variable word does not exist.
The basic idea to prove Theorem \ref{combth3} is to reduce
it to some simpler
cases. i.e., to construct some sequence $\h n_0\cdots \h n_{\h r}$
and colorings $\h f_{00},\h f_{01},\h f_{10},\h f_{11}:2^{\h N}\rightarrow 2$
with $\h n_0\cdots \h n_{\h r}, (\h f_{\v b}:\v b\in 2^2)$ preserving the otherwise hypothesis
while acquiring some additional restriction on
$(\h f_{\v b}:\v b\in 2^2)$.
More specifically,
 we firstly try to find
a $\v a_0 $ with $\v a_0^{-1}(1)\subseteq
  \cup_{s\leq r/2} N_s$
such that
$$
f_{00}(\v a_0) = f_{01}(\v a_0) = 0.
$$

If  $\v a_0$ exists, then consider
the sequence $n_{r/2+1}\cdots n_r$ and
the colorings
$$
\h f_{\v b}: 2^{n_{r/2+1}+\cdots +n_r}
\ni \v a\mapsto f_{\v b}((\v a_0\uhr n)^{\smallfrown} \v a)
$$
where $n = n_0+\cdots +n_r$.
Note that
the otherwise hypothesis is preserved by
the sequence $n_{r/2+1}\cdots n_r$
and the colorings $(\h f_{\v b}:\v b\in 2^2)$.
i.e., for all $\v b\in 2^2$,
$n_{r/2+1}\cdots n_r$ is not \sectionhomo\ for $\h f_{\v b}$;
and there does not exist a $\02$-variable word of length $n_{r/2+1}+\cdots +n_r$
such that $\h f_{\v b}(v(\v b)) = \h f_{\v{\h b}}(v(\v{\h b}))$
for all $\v b,\v{\h b}\in 2^2$
\footnote{Also note that $n_{r/2+1}\cdots n_r$
is still sufficiently long in the sense of whatever needed in the following proof
since $n_0\cdots n_r$ is sufficiently long.}.
In addition,
 for every $\v a\in 2^{n_{r/2+1}+\cdots +n_r }$,

 \tikzcdset{every label/.append style = {font = \tiny}}
\begin{equation}\label{combeq24}
\neg (\h f_{10}(\v a) = \h f_{11}(\v a)=0 )
\footnote{
We use the diagram beside this formula to assist
the reader to remember the restriction.
The label ``0" on the edge $(f_{10},f_{11})$
refers to the fact that
$\neg (\h f_{10}(\v a) = \h f_{11}(\v a)=0 )$ for all $\v a\in 2^N$.}.
\hspace{1cm}
\begin{tikzcd}[column sep = 0.3em, row sep = 0.2em, inner xsep = 0.02em]
 & \text{\tiny$ f_{11} $}\arrow[" " description, dl, dash]\arrow[dr, dash, "0" description] &
\\
 \text{\tiny$ f_{01} $} &                        &  \text{\tiny$ f_{10} $}\\
&  \text{\tiny$ f_{00} $}\arrow[ul,dash, " "  description]\arrow[ur, dash," "  description]&
\end{tikzcd}
\end{equation}
To see this, suppose $\h f_{10}(\v a) = \h f_{11}(\v a) = 0$.
Consider the $1$-variable word $v$
of length $N$ such that $v(00) = v(01) = \v a_0$,
$v(10) = v(11) = (\v a_0\uhr n)^{\smallfrown} \v a$,
we have that
$f_{\v b}(v(\v b)) = 0$ for all $\v b\in 2^2$.
Thus, if $\v a_0 $ exists, we proceed the proof
with the sequence $n_{r/2+1}\cdots n_r$ (which will still be sufficiently long)
and the colorings $(\h f_{\v b}:\v b\in 2^2)$.
The good thing is that not only do we still have the otherwise hypothesis,
we also have the additional (\ref{combeq24}).

If  $\v a_0$ does not exist,
we proceed the proof with the sequence
$n_0\cdots n_{r/2}$
and  the
colorings $$
\h f_{\v b}:2^{n_0+\cdots+ n_{r/2}}\ni \v a\mapsto
f_{\v b}(\v a 0\cdots 0).
$$
Note that the otherwise hypothesis
is preserved. In addition, since  $\v a_0$ does not exist,
 for every $\v a\in 2^{ n }$,
$$
\neg \h f_{00}(\v a) = \h f_{01}(\v a)=0.
\hspace{1cm}
\begin{tikzcd}[column sep = 0.3em, row sep = 0.2em, inner xsep = 0.02em]
 & \text{\tiny$ f_{11} $}\arrow[" " description, dl, dash]\arrow[dr, dash, " " description] &
\\
 \text{\tiny$ f_{01} $} &                        &  \text{\tiny$ f_{10} $}\\
&  \text{\tiny$ f_{00} $}\arrow[ul,dash, "0 "  description]\arrow[ur, dash," "  description]&
\end{tikzcd}
$$

Whichever is the case, we preserve the
otherwise hypothesis and acquire an additional
restriction on the behavior of the  colorings.
Suppose it is the case of (\ref{combeq24}),
for simplicity, denote $n_{r+1}\cdots n_r, \h f_{\v b}$
as $n_0\cdots n_r, f_{\v b}$ respectively.
Apply the above argument again (with respect to the new sequence
and colorings) against
the    edges
 $(f_{00},f_{10})$, $(f_{01},f_{11})$ and color $0$
 (where the above argument is against
 $(f_{00},f_{01}), (f_{10},f_{11})$ and color $0$
 and results in some sequence and colorings so that
 the joint behavior of either $(f_{00},f_{01})$
 or $(f_{10},f_{11})$ is further restricted).
 We have that there exists a sequence
 $\h n_0\cdots \h n_{\h r}$
 and colorings $\h f_{00},\h f_{01},\h f_{10},\h f_{11}:
 2^{\h N}\rightarrow 2$
 preserving the otherwise hypothesis
 and
 (\ref{combeq24}). In addition
 \begin{eqnarray}
 \text{either }
 &\neg \h f_{00}(\v a)= \h f_{10}(\v a)=0\text{ for all }\v a\in 2^{\h N};
 \hspace{1cm}
\begin{tikzcd}[column sep = 0.3em, row sep = 0.2em, inner xsep = 0.02em]
 & \text{\tiny$ f_{11} $}\arrow[" " description, dl, dash]\arrow[dr, dash, " " description] &
\\
 \text{\tiny$ f_{01} $} &                        &  \text{\tiny$ f_{10} $}\\
&  \text{\tiny$ f_{00} $}\arrow[ul,dash, "0 "  description]\arrow[ur, dash,"0"  description]&
\end{tikzcd}
 \\ \nonumber
 \text{ or }&\neg \h f_{01}(\v a)= \h f_{11}(\v a)=0\text{ for all }\v a\in 2^{\h N};
  \hspace{1cm}
 \begin{tikzcd}[column sep = 0.3em, row sep = 0.2em, inner xsep = 0.02em]
 & \text{\tiny$ f_{11} $}\arrow["0" description, dl, dash]\arrow[dr, dash, " " description] &
\\
 \text{\tiny$ f_{01} $} &                        &  \text{\tiny$ f_{10} $}\\
&  \text{\tiny$ f_{00} $}\arrow[ul,dash, "0 "  description]\arrow[ur, dash," "  description]&
\end{tikzcd}
 \end{eqnarray}
Keeping applying this argument for each pair of
non adjunctive edges and each color, we have the following.

\begin{lemma}\label{comblemred}
Let $n_0\cdots n_r$ be sufficiently long and is
not \sectionhomo\ for the  colorings
$f_{00},f_{10},f_{01},f_{11}:2^N\rightarrow 2$.
Suppose  there does not exist a $\02$-variable word $v$
such that for every $\v b,\v{\h b}\in 2^2$,
$f_{\v b}(v(\v b)) = f_{\v{\h b}}(v(\v{\h b}))$.
Then there exists a sufficiently long sequence
$\h n_0\cdots \h n_{\h r}$ and colorings $\h f_{00},\h f_{10},\h f_{01},\h f_{11}:
2^{\h N}\rightarrow 2$ such that the hypothesis of this lemma is preserved.
In addition, one of the following is true:
\begin{itemize}
\item
for every $\v a\in 2^{\h N}$,
\begin{eqnarray}\label{combeq30}
&\neg \h f_{01}(\v a) = \h f_{11}(\v a) = 0, \neg \h f_{10}(\v a) = \h f_{11}(\v a) = 0.
  \hspace{1cm}
 \begin{tikzcd}[column sep = 0.3em, row sep = 0.2em, inner xsep = 0.02em]
 & \text{\tiny$ f_{11} $}\arrow["0" description, dl, dash]\arrow[dr, dash, "0" description] &
\\
 \text{\tiny$ f_{01} $} &                        &  \text{\tiny$ f_{10} $}\\
&  \text{\tiny$ f_{00} $}\arrow[ul,dash, " "  description]\arrow[ur, dash," "  description]&
\end{tikzcd}
\end{eqnarray}

\item for every $\v a\in 2^{\h N}$,
\begin{eqnarray}\label{combeq31}
&\neg \h f_{00}(\v a) = \h f_{01}(\v a) , \neg\h f_{01}(\v a) = \h f_{11}(\v a) .
  \hspace{1cm}
 \begin{tikzcd}[column sep = 0.3em, row sep = 0.2em, inner xsep = 0.02em]
 & \text{\tiny$ f_{11} $}\arrow["01" description, dl, dash]\arrow[dr, dash, " " description] &
\\
 \text{\tiny$ f_{01} $} &                        &  \text{\tiny$ f_{10} $}\\
&  \text{\tiny$ f_{00} $}\arrow[ul,dash, "01"  description]\arrow[ur, dash," "  description]&
\end{tikzcd}
\end{eqnarray}

\item for every $\v a\in 2^{\h N}$,
\begin{eqnarray}\label{combeq32}
&\neg \h f_{00}(\v a) = \h f_{01}(\v a)=0 , \neg \h f_{00}(\v a) = \h f_{10}(\v a) = 1,&
\\ \nonumber
&\neg \h f_{01}(\v a) =\h  f_{11}(\v a) = 0, \neg \h f_{10}(\v a) = \h f_{11}(\v a) = 1.&
  \hspace{1cm}
  \begin{tikzcd}[column sep = 0.3em, row sep = 0.2em, inner xsep = 0.02em]
 & \text{\tiny$ f_{11} $}\arrow["0" description, dl, dash]\arrow[dr, dash, "1" description] &
\\
 \text{\tiny$ f_{01} $} &                        &  \text{\tiny$ f_{10} $}\\
&  \text{\tiny$ f_{00} $}\arrow[ul,dash, "0"  description]\arrow[ur, dash,"1"  description]&
\end{tikzcd}
\end{eqnarray}
 \end{itemize}
\end{lemma}
\begin{proof}

The point is that for every pair of non adjunctive edges
say $(f_{00},f_{01}), (f_{10},f_{11})$, every color $j\in 2$,
$j$ appears on one of the  edges.

Clearly, some cases are reduced to one of the three cases due to symmetry.
For example, we may obtain
$\h n_0\cdots \h n_{\h r}$ and colorings $\h f_{00},\h f_{10},\h f_{01},\h f_{11}$
so that for every $\v a\in 2^{\h N}$,
\begin{eqnarray}\nonumber
&\neg \h f_{01}(\v a) = \h f_{11}(\v a) = 1, \neg \h f_{10}(\v a) = \h f_{11}(\v a) = 1.
 \hspace{1cm}
 \begin{tikzcd}[column sep = 0.3em, row sep = 0.2em, inner xsep = 0.02em]
 & \text{\tiny$ f_{11} $}\arrow[ "1" description, dl, dash]\arrow[dr, dash, "1" description] &
\\
 \text{\tiny$ f_{01} $} &                        &  \text{\tiny$ f_{10} $}\\
&  \text{\tiny$ f_{00} $}\arrow[ul,dash, " "  description]\arrow[ur, dash, " "  description]&
\end{tikzcd}
\end{eqnarray}
This clearly reduces to (\ref{combeq30}) by considering the colorings
$\t f_{\v b} = 1- \h f_{\v b}$.

Other two symmetry cases are less trivial but still easy.
One of the symmetry case is
for every $\v a\in 2^{\h N}$,
\begin{eqnarray}\nonumber
&\neg \h f_{00}(\v a) = \h f_{10}(\v a), \neg \h f_{10}(\v a) = \h f_{11}(\v a) .
  \hspace{1cm}
 \begin{tikzcd}[column sep = 0.3em, row sep = 0.2em, inner xsep = 0.02em]
 & \text{\tiny$ f_{11} $}\arrow[" " description, dl, dash]\arrow[dr, dash, "01" description] &
\\
 \text{\tiny$ f_{01} $} &                        &  \text{\tiny$ f_{10} $}\\
&  \text{\tiny$ f_{00} $}\arrow[ul,dash, " "  description]\arrow[ur, dash, "01"  description]&
\end{tikzcd}
\end{eqnarray}
This can be reduced to (\ref{combeq31})
by considering  $\t f_{bb} = \h f_{bb}, \t f_{b\cdot 1-b} = \h f_{1-b\cdot b}$.
Note that for every $\02$-variable word $v$,
if $\t f_{\v b}(v(\v b)) = j$ for some $j$ for all $\v b\in 2^2$,
then consider the variable word $\h v$ such that
$\h v(bb) = v(bb), \h v(b\cdot 1-b) = v(1-b\cdot b)$,
we have that $\h f_{\v b}(v(\v b)) = j$ for all $\v b\in 2^2$.

Another symmetry case is that, say
for every $\v a\in 2^{\h N}$,
\begin{eqnarray}\label{combeq33}
&\neg \h f_{00}(\v a) = \h f_{01}(\v a)=0, \neg \h f_{00}(\v a) = \h f_{10}(\v a)=0 .
 \hspace{1cm}
 \begin{tikzcd}[column sep = 0.3em, row sep = 0.2em, inner xsep = 0.02em]
 & \text{\tiny$ f_{11} $}\arrow["" description, dl, dash]\arrow[dr, dash, "" description] &
\\
 \text{\tiny$ f_{01} $} &                        &  \text{\tiny$ f_{10} $}\\
&  \text{\tiny$ f_{00} $}\arrow[ul,dash, "0"  description]\arrow[ur, dash,"0"  description]&
\end{tikzcd}
\end{eqnarray}
This can be reduced to (\ref{combeq30}).
Let $h(\v a) = \v{\overline{ a}}$ be such that $\v{\overline{ a}}(t) = 1- \v a(t)$
for all $t< |\v a|$. Consider $\t f_{\v b} = \h f_{\v{\overline{b}}}\circ h$.
Where $\v{\overline{ b}}(t) = 1-\v b(t)$.
Now (\ref{combeq33})
transforms into (\ref{combeq30}) with respect to $\t f_{\v b}$
and clearly $\h n_0\cdots \h n_{\h r}$ is not \sectionhomo\ for $\t f_{\v b}$ for all $\v b\in 2^2$.
On the other hand, suppose for some $\02$-variable word $v$, $\t f_{\v b}(v(\v b)) = j$
for all $\v b\in 2^2$.
Consider the variable word $\overline{v}$ such that
$\overline{v}(t) = 1-v(t)$ if $v(t)\in 2$
and $\overline{v}(t) = v(t)$ if $v(t) $ is a variable.
We have that
$$
\h f_{\v b}( \overline{v}(\v b)) =
\t f_{\v{\overline{b}}}(h^{-1}(\overline{v}(\v b)))
 = \t f_{\v{\overline{b}}}
 (v(\v{\overline{b}}))=j.
$$

\end{proof}

By Lemma \ref{comblemred},
to prove Theorem \ref{combth3}, it remains to derive a contradiction for
each of the three
cases.
We start with   case (\ref{combeq30})
(which might be the simplest case)
in section \ref{combsubsec0}
and the other cases in section \ref{combsubsec1},\ref{combsubsec2}
respectively.
Each case of (\ref{combeq31})(\ref{combeq32})  is further
reduced to some $\h n_0\cdots \h n_{\h r},
(\h f_{\v b}:\v b\in 2^2)$
preserving the case hypothesis while
imposing further restrictions on the behavior
of the colorings.
Using the diagram symbol,
the three cases proceed as following:
\begin{eqnarray}\nonumber
&\begin{tikzcd}[column sep = 0.3em, row sep = 0.2em, inner xsep = 0.02em]
 & \text{\tiny$ f_{11} $}\arrow["0" description, dl, dash]\arrow[dr, dash, "0" description] &
\\
 \text{\tiny$ f_{01} $} &                        &  \text{\tiny$ f_{10} $}\\
&  \text{\tiny$ f_{00} $}\arrow[ul,dash, " "  description]\arrow[ur, dash," "  description]&
\end{tikzcd}
&\overset{Lemma \ref{comblem3}}{\Rightarrow}
\text{contradiction};
\\ \nonumber
&\begin{tikzcd}[column sep = 0.3em, row sep = 0.2em, inner xsep = 0.02em]
 & \text{\tiny$ f_{11} $}\arrow["01" description, dl, dash]\arrow[dr, dash, " " description] &
\\
 \text{\tiny$ f_{01} $} &                        &  \text{\tiny$ f_{10} $}\\
&  \text{\tiny$ f_{00} $}\arrow[ul,dash, "01"  description]\arrow[ur, dash," "  description]&
\end{tikzcd}
&\overset{Lemma \ref{comblem4}}{\Rightarrow}
\begin{tikzcd}[column sep = 0.3em, row sep = 0.2em, inner xsep = 0.02em]
 & \text{\tiny$ f_{11} $}\arrow["01" description, dl, dash]\arrow[dr, dash, " " description] &
\\
 \text{\tiny$ f_{01} $} &      \text{\tiny $01$}                    &  \text{\tiny$ f_{10} $}\\
&  \text{\tiny$ f_{00} $}\arrow[ul,dash, "01"  description]\arrow[ur, dash," "  description]&
\end{tikzcd}
\overset{Lemma \ref{comblem6}}{\Rightarrow}
\text{contradiction};\\ \nonumber
&\begin{tikzcd}[column sep = 0.3em, row sep = 0.2em, inner xsep = 0.02em]
 & \text{\tiny$ f_{11} $}\arrow["0" description, dl, dash]\arrow[dr, dash, "1" description] &
\\
 \text{\tiny$ f_{01} $} &                        &  \text{\tiny$ f_{10} $}\\
&  \text{\tiny$ f_{00} $}\arrow[ul,dash, "0"  description]\arrow[ur, dash,"1"  description]&
\end{tikzcd}
&\overset{Lemma \ref{comblem8}}{\Rightarrow}
\begin{tikzcd}[column sep = 0.3em, row sep = 0.2em, inner xsep = 0.02em]
 & \text{\tiny$ f_{11} $}\arrow["0" description, dl, dash]\arrow[dr, dash, "1" description] &
\\
 \text{\tiny$ f_{01} $} &      \text{\tiny $01$}             &  \text{\tiny$ f_{10} $}\\
&  \text{\tiny$ f_{00} $}\arrow[ul,dash, "0"  description]\arrow[ur, dash,"1"  description]&
\end{tikzcd}
\overset{Lemma \ref{comblem10}}{\Rightarrow}
\text{contradiction}.
\end{eqnarray}

\subsubsection{Case (\ref{combeq30})}
\label{combsubsec0}
The following lemma serves as a warm up and its idea will
be further used.
\begin{lemma}\label{comblem2}
Let $n_0\cdots n_r$ be sufficiently long and is
not \sectionhomo\ for  coloring $f:2^N\rightarrow 2$.
Then there exist    $\v a_0,\v a_1\in 2^N$ with $\v a_0 \cap \v a_1 =\emptyset$
such that $  f(\v a_0)=f(\v a_1)=0, f(\v a_0\cup \v a_1) = 1$.

\end{lemma}
\begin{proof}
Suppose otherwise.
It suffices to show that there exist
$\v a_0,\cdots ,\v a_{n_0+n_1-1}$ with
$\v a_m^{-1}(1)$ being mutually disjoint
and $\min \v a^{-1}_m(1) = m$ for all $m<n_0+n_1$
such that
\begin{eqnarray}\label{combeq8}
\text{for every nonempty set $I\subseteq n_0+n_1$,
$f(\cup_{m\in I}\v a_m ) = 0$.
}
\end{eqnarray}
Because this implies $n_0\cdots n_r$ is \sectionhomo\ for $f$,
a contradiction.

We define $\v a_0,\cdots ,\v a_{n_0+n_1-1}$ by induction.
Let $\v a_0$ satisfy $0\in \v a_0^{-1}(1)\subseteq N_2\cup \{0\}$
(where $N_s = \{n_0+\cdots +n_{s-1},\cdots, n_0+\cdots +n_s-1\}$) and
$f(\v a_0) = 0$
(which must exist since $n_0\cdots n_r$ is not \sectionhomo\ for $f$ and by \search\ argument).
Note that by the otherwise hypothesis,
\begin{align}\label{combeq9}
\text{ for every $\v a\in 2^N$ with
$\v a \cap \v a_0  = \emptyset$ and with $f(\v a) = 0$,
we have $f(\v a\cup \v a_0) = 0$.
}
\end{align}
Let $\v a_1$ satisfy $1\in \v a_1^{-1}(1)\subseteq N_3\cup\{1\}$
and $f(\v a_1) = 0$; and let $\v a_2,\cdots,\v a_{n_0+n_1-1}$ be defined
in the same fashion.
We verify (\ref{combeq8}) by induction on $|I|$.
For  $I\subseteq n_0+n_1$ with $|I| = 1$ the conclusion clearly follows
by definition of $(\v a_m:m<n_0+n_1)$.
Suppose the conclusion follows when $|I|\leq n$ and now $|I| = n+1$
with $m\in I$.
Consider $\v a = \cup_{\h m\in I\setminus \{m\}}\v a_{\h m}$
and note that $\v a \cap \v a_m=\emptyset$.
By induction, $f(\v a) = 0$.
Therefore, by (\ref{combeq9}),
$f(\v a\cup \v a_m) = 0$. Thus  we are done.

\end{proof}

Note that for Lemma \ref{comblem2}, to be sufficiently long,
it suffices that $r-1\geq n_0+n_1$.

\begin{lemma}\label{comblem3}
Let $n_0\cdots n_r$ be sufficiently long and is
not \sectionhomo\ for  colorings
$f_{00},f_{10},f_{01},f_{11}:2^N\rightarrow 2$.
Suppose  there does not exist a $\02$-variable word $v$
such that for every $\v b,\v{\h b}\in 2^2$,
$f_{\v b}(v(\v b)) = f_{\v{\h b}}(v(\v{\h b}))$.
Suppose for every $\v a\in 2^N$,
\begin{eqnarray}\label{combeq10}
&\neg f_{01}(\v a) = f_{11}(\v a) = 0, \neg f_{10}(\v a) = f_{11}(\v a) = 0.
\hspace{1cm}\begin{tikzcd}[column sep = 0.3em, row sep = 0.2em, inner xsep = 0.02em]
 & \text{\tiny$ f_{11} $}\arrow["0" description, dl, dash]\arrow[dr, dash, "0" description] &
\\
 \text{\tiny$ f_{01} $} &                        &  \text{\tiny$ f_{10} $}\\
&  \text{\tiny$ f_{00} $}\arrow[ul,dash, " "  description]\arrow[ur, dash," "  description]&
\end{tikzcd}
\end{eqnarray}
Then there is a contradiction.

\end{lemma}
\begin{proof}

For the sake of a contradiction,
it suffices to show that there exists a
$\02$-variable word $v$ such that
\begin{eqnarray}
\label{combeq25}f_{00}(v(00)) = f_{11}(v(11)) = 1,
f_{11}(v(01)) = f_{11}(v(10)) = 0
\end{eqnarray}
since by (\ref{combeq10}), this implies that
$f_{\v b}(v(\v b)) = 1$ for all $\v b\in 2^2$.
Without loss of generality, assume that
$  f_{00}(0\cdots 0) = 1$
otherwise let
(by \search\ argument) $\v a\in 2^N$ satisfy $\v a^{-1}(1)\subseteq N_0$,
$\h f_{00}(\v a)=1$ and proceed the proof with
the sequence $n_1n_2\cdots n_r$ and the four colorings
$
\t f_{\v b}:2^{N-n_0}\ni \v{\h a}\mapsto
  f_{\v b}((\v a\uhr n_0)^{\smallfrown}\v{\h a}).
$
Since $n_0\cdots n_r$ is sufficiently long,
by Lemma \ref{comblem2}, let $\v a_0,\v a_1\in 2^{N }$
be such that
 $\v a_0 \cap \v a_1=\emptyset$ and
\begin{eqnarray}\label{combeq22}
 f_{11}(\v a_0) =     f_{11}(\v a_1) = 0,
   f_{11}(\v a_0\cup \v a_1) = 1.
   \end{eqnarray}
  Clearly there is a $\02$-variable word $v$
such that
\begin{eqnarray}\label{combeq36}
v(00) = 0\cdots 0,\ v(01) =  \v a_0
, v(10) =  \v a_1
, v(11) =   \v a_0\cup \v a_1 .
\end{eqnarray}
Thus (\ref{combeq22})(\ref{combeq36})
and ``$f_{00}(0\cdots 0)=1$"
together  verify
(\ref{combeq25}). So we are done.

\end{proof}
Note that for
Lemma \ref{comblem3},
to be sufficiently long,
it suffices that $r-2\geq n_1+n_2$.

\subsubsection{Case (\ref{combeq31})}
\label{combsubsec1}
In the following lemma,
we further reduce
case (\ref{combeq31})
to some $\h n_0\cdots \h n_{\h r}$ and colorings $\h f_{00},\h f_{10},\h f_{01},\h f_{11}:
2^{\h N}\rightarrow 2$
so that some additional restriction
is imposed.
\begin{lemma}\label{comblem4}
Let $n_0\cdots n_r$ be sufficiently long and is
not \sectionhomo\ for  colorings
$f_{00},f_{10},f_{01},f_{11}:2^N\rightarrow 2$.
Suppose there does not exist a $\02$-variable word
$v$ such that $f_{\v b}(v(\v b)) = f_{\v{\h b}}(v(\v{\h b}))$
for all $\v b,\v{\h b}\in 2^2$.
Suppose for every $\v a\in 2^N$,
\begin{eqnarray}\label{combeq13}
&\neg f_{00}(\v a) = f_{01}(\v a) , \neg f_{01}(\v a) = f_{11}(\v a) .
\hspace{1cm}
\begin{tikzcd}[column sep = 0.3em, row sep = 0.2em, inner xsep = 0.02em]
 & \text{\tiny$ f_{11} $}\arrow["01" description, dl, dash]\arrow[dr, dash, " " description] &
\\
 \text{\tiny$ f_{01} $} &                        &  \text{\tiny$ f_{10} $}\\
&  \text{\tiny$ f_{00} $}\arrow[ul,dash, "01"  description]\arrow[ur, dash," "  description]&
\end{tikzcd}
\end{eqnarray}
Then there exists a sufficiently long sequence
$\h n_0\cdots \h n_{\h r}$ and colorings $\h f_{00},\h f_{10},\h f_{01},\h f_{11}:
2^{\h N}\rightarrow 2$ such that the hypothesis of this lemma is preserved.
\text{In addition, }
\begin{eqnarray}\nonumber
&\neg \h f_{10}(\v a) = \h f_{01}(\v a)
\text{ for all $\v a\in 2^{\h N}$}.
\hspace{1cm}
\begin{tikzcd}[column sep = 0.3em, row sep = 0.2em, inner xsep = 0.02em]
 & \text{\tiny$ f_{11} $}\arrow["01" description, dl, dash]\arrow[dr, dash, " " description] &
\\
 \text{\tiny$ f_{01} $} &      \text{\tiny$ 01 $}                  &  \text{\tiny$ f_{10} $}\\
&  \text{\tiny$ f_{00} $}\arrow[ul,dash, "01"  description]\arrow[ur, dash," "  description]&
\end{tikzcd}
\end{eqnarray}

\end{lemma}
\begin{proof}
It suffices to show that
there exists a sufficiently long sequence
$\h n_0\cdots \h n_{\h r}$ and colorings $\h f_{00},\h f_{10},\h f_{01},\h f_{11}:
2^{\h N}\rightarrow 2$
such that the hypothesis of this lemma is preserved and
in addition,
 for every $\v a\in 2^{\h N}$,
\begin{eqnarray}\label{combeq12}
&\neg \h f_{10}(\v a) = \h f_{01}(\v a)=0.
\end{eqnarray}
The conclusion follows by applying the proof of (\ref{combeq12}) twice.
Without loss of generality, assume that
$f_{00}(0\cdots 0) = 0$
otherwise let $\v a\in 2^N$ satisfy $\v a^{-1}(1)\subseteq N_0$,
$f_{00}(\v a)=0$ and proceed the proof with the four colorings
$
\t f_{\v b}:2^{N-n_0}\ni \v{\h a}\mapsto
f_{\v b}((\v a\uhr n_0)^{\smallfrown}\v{\h a})
$
(note that the hypothesis of this lemma holds for $(\t f_{\v b}: \v b\in 2^2)$).
Assume   otherwise that
\begin{eqnarray}
\label{combeq21}\text{the sequence and the four colorings
as (\ref{combeq12}) do not exist.}
\end{eqnarray}
\begin{claim}\label{combclaim2}
There exist $\v a_0,\cdots, \v a_{n_0+n_1-1}\in 2^N$
with $  \v a_m^{-1}(1)$ being mutually disjoint such that
for every $m<n_0+n_1$, $\min \v a_m^{-1}(1) = m$ and
 $
f_{01}(\v a_m) = f_{10}(\v a_m) = 0.
 $
\end{claim}
\begin{proof}
This is done  by \search\ argument.
Let $\{(n_s: s\in J_m)\}_{m<n_0+n_1}$
be $n_0+n_1$ many mutually disjoint subsequences of $n_0\cdots n_r$
(meaning $J_m$ are mutually disjoint) with $0,1\notin J_m$
such that each is sufficiently long.
For each $m<n_0+n_1$, there must be a $\v a_m$ with $m\in  
\v a_m^{-1}(1)\subseteq (\cup_{s\in J_m}N_s)\cup \{m\} $
such that $f_{01}(\v a_m) = f_{10}(\v a_m) = 0$.
Otherwise, suppose  say  $\v a_0$ does not exist.
For convenience, suppose $J_0 = \{2,3,\cdots, \t r\}$.
Then consider the sequence $ (n_s: s\in J_0) $
and  the colorings
$$
\h f_{\v b}: 2^{n_2+\cdots+ n_{\t r}}\ni \v{\h a}\mapsto f_{\v b}(10_{n_0+n_1-1}\v{\h a} 0\cdots 0).
$$
We have that $(n_s:s\in J_0)$
and $(\h f_{\v b}: \v b\in 2^2)$
 are as desired in (\ref{combeq12}), a contradiction to (\ref{combeq21}).

\end{proof}
It suffices to show that
\begin{claim}\label{combclaim0}
\text{For every nonempty $I\subseteq n_0+n_1$,
$f_{11}(\cup_{m\in I}\v a_m) = 1$}.
\end{claim}
\noindent Because this implies that $n_0\cdots n_r$ is \sectionhomo\ for $f_{11}$,
a contradiction.
\begin{proof}
We prove Claim \ref{combclaim0} similarly as in the proof of Lemma \ref{comblem2}.
When $|I| = 1$, the conclusion clearly follows since
$f_{11}(\v a_m)\ne f_{01}(\v a_m) = 0$ (by (\ref{combeq13})).
The point is,
\begin{align}\label{combeq16}
&\text{for every $m<n_0+n_1$, every $\v a\in 2^N$ with
$\v a\cap \v a_m=\emptyset$, }\\ \nonumber
&
\text{and with $f_{11}(\v a) = 1$,}
\text{we have $f_{11}(\v a\cup \v a_m) = 1$}.
\end{align}
To see this, note that $f_{11}(\v a) = 1$
implies $f_{01}(\v a) = 0$.
By definition of $\v a_m$, $f_{10}(\v a_m) = 0$
and recall that $f_{00}(0\cdots 0) = 0$.
Therefore, if on the other hand, $f_{11}(\v a\cup \v a_m) = 0$,
then let $v$ be the $\02$-variable word such that
$$v(00) = 0\cdots 0, v(01) = \v a, v(10) = \v a_m, v(11) = \v a\cup \v a_m,$$
we have $f_{\v b}(v(\v b)) = 0$ for all $\v b\in 2^2$, a contradiction
to the hypothesis of this lemma.

Suppose the claim holds when $|I|\leq n$ and now $|I| = n+1$
with $m\in I$.
Consider $\v a = \cup_{\h m\in I\setminus \{m\}}\v a_{\h m}$
and note that $\v a \cap \v a_m =\emptyset$.
By induction, $f_{11}(\v a) = 1$.
Therefore, by (\ref{combeq16}),
$f_{11}(\v a\cup \v a_m) = 1$. Thus  we are done.
\end{proof}

\end{proof}

In Lemma \ref{comblem4}, to be sufficiently long, it suffices
that there are $n_1+n_2$ many mutually disjoint subsequences
$\{(n_s:s\in J_{m})\}_{m<n_1+n_2}$ (meaning
$J_m$ are mutually disjoint) of $n_0\cdots n_r$ with $0,1,2\notin J_m$
such that each subsequence, say $(n_s:s\in J_0) = \h n_0\cdots \h n_{\h r}$,
 is sufficiently long in the sense
that there are $\h n_0+\h n_1$ many mutually disjoint subsequences of $\h n_0\cdots \h n_{\h r}$
such that each is sufficiently long in the sense
of Lemma \ref{comblem6}
\footnote{Don't forget that the first section $n_0$ is used
in the ``without loss of generality assumption".}.
Before Lemma \ref{comblem6}, we need a lemma similar to Lemma \ref{comblem2}.
\begin{lemma}\label{comblem7}
Let $n_0\cdots n_r$ be sufficiently long and is
not \sectionhomo\ for  coloring $f:2^N\rightarrow 2$.
Then there exist    $\v a_0,\v a_1\in 2^N$ with $\v a_0\cap \v a_1=\emptyset$
such that $  \neg f(\v a_0)= f(\v a_1)$ and $ f(\v a_0\cup \v a_1) = 1$.

\end{lemma}
\begin{proof}
Assume otherwise.
By \search\ argument, let $\v a$ be such that $\v a^{-1}(1)\supseteq \cup_{s\leq \h r} N_s$
(where $n_0\cdots n_{\h r}$ is a sufficiently long subsequence of $n_0\cdots n_r$,
say $\h r\geq n_0$)
and $f(\v a) = 1$.
By \search\ argument, let $\v a_0\cdots \v a_{n_0-1}$ be such that
$\v a_m^{-1}(1)$ are mutually disjoint; and for every $m<n_0$,
$\min \v a_m^{-1}(1)=m, \v a_m^{-1}(1)\subseteq \v a^{-1}(1)$,
$f(\v a_m) = 1$
.
It suffices to show that for every $I\subseteq n_0$,
$$
f(\v a\setminus (\cup_{m\in I}\v a_m)) = 1
$$
since  this implies that $n_0\cdots n_r$ is \sectionhomo\ for $f$, a contradiction.
Since $f(\v a) = 1$, the conclusion follows when $|I| = 0$.
Suppose it follows when $|I|\leq n$ and now $|I| = n+1$
with $m\in I$.
Consider $$\v{\h a} = \v a\setminus (\cup_{\h m\in I\setminus \{m\}}\v a_{\h m}) ,
 \v{\h a}_0 = \v a_m,
\v{\h a}_1 = \v{\h a}\setminus \v{\h a}_0.$$
Note that $\v{\h a}_0\cap \v{\h a}_1=\emptyset$;
and $f(\v{\h a}_0\cup \v{\h a}_1) = f(\v{\h a}) = 1$
(by induction) and $f(\v{\h a}_0) = 1$ (by definition of $\v a_m$).
Thus by the otherwise assumption,
$f(\v{\h a}_1) = f(\v a\setminus (\cup_{m\in I}\v a_m)) = 1$
and we are done.

\end{proof}
In Lemma \ref{comblem7}, to be sufficiently long, it suffices that
$r\geq n_0+1$.
\begin{lemma}\label{comblem6}
Let $n_0\cdots n_r$ be sufficiently long and is
not \sectionhomo\ for  colorings
$f_{00},f_{10},f_{01},f_{11}:2^N\rightarrow 2$.
Suppose there does not exist a $\02$-variable word
$v$ such that $f_{\v b}(v(\v b)) = f_{\v{\h b}}(v(\v{\h b}))$
for all $\v b,\v{\h b}\in 2^2$.
Suppose for every $\v a\in 2^N$,
\begin{eqnarray}\label{combeq35}
&\neg f_{00}(\v a) = f_{01}(\v a) ,& \neg f_{01}(\v a) = f_{11}(\v a),\\ \nonumber
&\neg f_{01}(\v a) = f_{10}(\v a).&
\hspace{1cm}
\begin{tikzcd}[column sep = 0.3em, row sep = 0.2em, inner xsep = 0.02em]
 & \text{\tiny$ f_{11} $}\arrow["01" description, dl, dash]\arrow[dr, dash, " " description] &
\\
 \text{\tiny$ f_{01} $} &      \text{\tiny$ 01 $}                  &  \text{\tiny$ f_{10} $}\\
&  \text{\tiny$ f_{00} $}\arrow[ul,dash, "01"  description]\arrow[ur, dash," "  description]&
\end{tikzcd}
\end{eqnarray}
Then there is a contradiction.

\end{lemma}
\begin{proof}
We use (\ref{combeq35})
to show that a $\02$-variable word (as in the hypothesis of this lemma) exists, thus a contradiction.
Without loss of generality, suppose $f_{00}(0\cdots 0) = 0$.
By Lemma \ref{comblem7},
there exist $\v a_0,\v a_1$ with
$\v a_0\cap \v a_1=\emptyset$
such that $$f_{01}(\v a_0) = 0, f_{01}(\v a_1) = 1
\text{ and }
f_{01}(\v a_0\cup \v a_1) = 1.$$
By (\ref{combeq35}),
$$
f_{10}(\v a_1) = 0, f_{11}(\v a_0\cup \v a_1) = 0.
$$
Thus let $v$ be a $\02$-variable word such that
$v(\v b) = \cup_{i\in \v b} \v a_i$, we have
$f_{\v b}(v(\v b)) = 0$ for all $\v b\in 2^2$,
a contradiction.

\end{proof}

In Lemma \ref{comblem6}, to be sufficiently long, it suffices that
$r-1\geq n_1+1$.

\subsubsection{Case (\ref{combeq32})}
\label{combsubsec2}

This case is similar to (\ref{combeq31}).
\begin{lemma}\label{comblem8}
Let $n_0\cdots n_r$ be sufficiently long and is
not \sectionhomo\ for  colorings
$f_{00},f_{10},f_{01},f_{11}:2^N\rightarrow 2$.
Suppose there does not exist a $\02$-variable word
$v$ such that $f_{\v b}(v(\v b)) = f_{\v{\h b}}(v(\v{\h b}))$
for all $\v b,\v{\h b}\in 2^2$.
Suppose for every $\v a\in 2^N$,
\begin{eqnarray}\label{combeq17}
&\neg f_{00}(\v a) = f_{01}(\v a)=0 , \neg f_{00}(\v a) = f_{10}(\v a) = 1,&
\\ \nonumber
&\neg f_{01}(\v a) = f_{11}(\v a) = 0, \neg f_{10}(\v a) = f_{11}(\v a) = 1.&
\hspace{1cm}
\begin{tikzcd}[column sep = 0.3em, row sep = 0.2em, inner xsep = 0.02em]
 & \text{\tiny$ f_{11} $}\arrow["0" description, dl, dash]\arrow[dr, dash, "1" description] &
\\
 \text{\tiny$ f_{01} $} &                &  \text{\tiny$ f_{10} $}\\
&  \text{\tiny$ f_{00} $}\arrow[ul,dash, "0"  description]\arrow[ur, dash,"1"  description]&
\end{tikzcd}
\end{eqnarray}
Then there exists a sufficiently long sequence
$\h n_0\cdots \h n_{\h r}$ and colorings $\h f_{00},\h f_{10},\h f_{01},\h f_{11}:
2^{\h N}\rightarrow 2$ such that the hypothesis of this lemma is preserved.
\text{ In addition, }
\begin{eqnarray}\label{combeq18}
&\neg \h f_{01}(\v a) = \h f_{10}(\v a)
\text{ for all $\v a\in 2^{\h N}$}.
\hspace{1cm}
\begin{tikzcd}[column sep = 0.3em, row sep = 0.2em, inner xsep = 0.02em]
 & \text{\tiny$ f_{11} $}\arrow["0" description, dl, dash]\arrow[dr, dash, "1" description] &
\\
 \text{\tiny$ f_{01} $} &      \text{\tiny $01$}             &  \text{\tiny$ f_{10} $}\\
&  \text{\tiny$ f_{00} $}\arrow[ul,dash, "0"  description]\arrow[ur, dash,"1"  description]&
\end{tikzcd}
\end{eqnarray}

\end{lemma}
\begin{proof}
It suffices to show that
there exists a sufficiently long sequence
$\h n_0\cdots \h n_r$ and colorings $\h f_{00},\h f_{10},\h f_{01},\h f_{11}:
2^{\h N}\rightarrow 2$ preserving the hypothesis
and in addition,
\begin{eqnarray} \label{combeq19}
&\neg \h f_{01}(\v a) = \h f_{10}(\v a)=0
\text{ for all $\v a\in 2^{\h N}$}.
\end{eqnarray}
The conclusion follows by applying the proof of (\ref{combeq19}) twice.
Without loss of generality, assume that
$f_{00}(0\cdots 0) = 0$.
Assume   otherwise that the four colorings
as (\ref{combeq19}) does not exist.
Then, by \search\ argument
(as in Claim \ref{combclaim2}),
there exist $\v a_0,\cdots, \v a_{n_0+n_1-1}\in 2^N$
with $  \v a_m^{-1}(1)$ being mutually disjoint such that
for every $m<n_0+n_1$, $\min \v a_m^{-1}(1)=m$ and
$$
f_{01}(\v a_m) = f_{10}(\v a_m) = 0.
$$
It suffices to show that
\begin{claim}\label{combclaim1}
\text{For every nonempty $I\subseteq n_0+n_1$,
$f_{11}(\cup_{m\in I}\v a_m) = 1$}.
\end{claim}
\noindent Because this implies that $n_0\cdots n_r$ is \sectionhomo\ for $f_{11}$,
a contradiction.
\begin{proof}
We prove Claim \ref{combclaim1} in the same way as Claim  \ref{combclaim0}.
When $|I| = 1$, the conclusion clearly follows since
$f_{11}(\v a_m)\ne f_{01}(\v a_m) = 0$ (by (\ref{combeq17})).
The point is,
\begin{align}\label{combeq26}
&\text{for every $m<n_0+n_1$, every $\v a\in 2^N$ with
$\v a\cap \v a_m=\emptyset$, }\\ \nonumber
&\text{and with $f_{11}(\v a) = 1$, }
\text{we have $f_{11}(\v a\cup \v a_m) = 1$}.
\end{align}
To see this, note that $f_{11}(\v a) = 1$
implies $f_{10}(\v a) = 0$.
If on the other hand, $f_{11}(\v a\cup \v a_m) = 0$,
then let $v$ be the $\02$-variable word such that
$$v(00) = 0\cdots 0, v(10) = \v a, v(01) = \v a_m, v(11) = \v a\cup \v a_m,$$
we have $f_{\v b}(v(\v b)) = 0$ for all $\v b\in 2^2$, a contradiction
to the hypothesis of this lemma.

Suppose the claim holds when $|I|\leq n$ and now $|I| = n+1$
with $m\in I$.
Consider $\v a = \cup_{\h m\in I\setminus \{m\}}\v a_{\h m}$
and note that $\v a \cap \v a_m =\emptyset$.
By induction, $f_{11}(\v a) = 1$.
Therefore, by (\ref{combeq26}),
$f_{11}(\v a\cup \v a_m) = 1$. Thus  we are done.
\end{proof}

\end{proof}

In Lemma \ref{comblem8}, the definition of sufficiently long
is similar as that of Lemma \ref{comblem4}.

\begin{lemma}\label{comblem10}
Let $n_0\cdots n_r$ be sufficiently long and is
not \sectionhomo\ for  colorings
$f_{00},f_{10},f_{01},f_{11}:2^N\rightarrow 2$.
Suppose for every $\v a\in 2^N$,
\begin{eqnarray}\label{combeq20}
&\neg   f_{00}(\v a) =  f_{01}(\v a)=0 ,& \neg   f_{00}(\v a) =   f_{10}(\v a) = 1,\\ \nonumber
&\neg   f_{01}(\v a) =   f_{11}(\v a) = 0,& \neg   f_{10}(\v a) =   f_{11}(\v a) = 1,
\\ \nonumber
&\neg   f_{01}(\v a) =   f_{10}(\v a).\ \ \ &
\hspace{1cm}
\begin{tikzcd}[column sep = 0.3em, row sep = 0.2em, inner xsep = 0.02em]
 & \text{\tiny$ f_{11} $}\arrow["0" description, dl, dash]\arrow[dr, dash, "1" description] &
\\
 \text{\tiny$ f_{01} $} &      \text{\tiny $01$}             &  \text{\tiny$ f_{10} $}\\
&  \text{\tiny$ f_{00} $}\arrow[ul,dash, "0"  description]\arrow[ur, dash,"1"  description]&
\end{tikzcd}
\end{eqnarray}
Then there is a contradiction.

\end{lemma}

\begin{proof}
This is simply because
given $\v a\in 2^N$, the only  value
of $f_{01}(\v a), f_{10}(\v a)$ consistent with (\ref{combeq20})
is
$$
f_{01}(\v a) = 1, f_{10}(\v a) = 0.
$$
Thus $n_0\cdots n_r$ is \sectionhomo\ for $f_{01},f_{10}$.

\end{proof}
In Lemma \ref{comblem10}, to be sufficiently long,
it suffices
that $r\geq0$.

\subsection{Questions and further discussion}

There are mainly two difficulties  
to generalize  Theorem \ref{combth2} (say, to $n_0=3$).
(1) The restriction given by an analog of Lemma \ref{comblemred}
is not strong enough;
(2) Even if $(f_{\v b}: \v b\in 2^{n_0})$
 satisfies a very strong restriction, say for some coloring
 $f:2^N\rightarrow 2$, $f_{\v b}(\v a) = f(\v a)+ |\v b^{-1}(1)|(\mod 2)$,
 the corresponding Lemma \ref{comblem2} is unknown:
 \begin{question}
Let $n>0$; let $n_0\cdots n_r$ be sufficiently long and is
not \sectionhomo\ for  coloring $f:2^N\rightarrow 2$.
Does  there exist a $[0,n]$-variable word $v$ and an $i\in 2$
such that $f(v(\v b)) = |\v b^{-1}(1)|+i (\mod 2)$ for all $\v b\in 2^n$.
 \end{question}

\noindent Actually we do not even known the following:
\begin{question}
Let $n_0\cdots n_r$ be sufficiently long and is
not \sectionhomo\ for  coloring $f:2^N\rightarrow 2$.
Does  there exist a $[0,3]$-variable word $v$ and an $i\in 2$
such that $f(v(\v b)) = |\v b^{-1}(1)|+i (\mod 2)$
for all $\v b\in 2^3$ with $|\v b^{-1}(1)|\in \{1,2\}$.
\end{question}
\noindent On the other hand, we don't know if Theorem \ref{combth3} follows 
with more than two colors. Actually we don't even know whether 
Lemma \ref{comblem2} follows with three colors:
\begin{question}
Let $n_0\cdots n_r$ be sufficiently long and is
not \sectionhomo\ for  coloring $f:2^N\rightarrow 3$.
Does there exist  $\v a_0,\v a_1\in 2^N$ with $\v a_0 \cap \v a_1 =\emptyset$
such that $  f(\v a_0)=f(\v a_1)=0, f(\v a_0\cup \v a_1) = 1$.
\end{question}

In Theorem \ref{combth3}, to be sufficiently long means containing 
a subsequence that is sufficiently long in a weaker sense; and to be
sufficiently long in that weaker sense means containing
a subsequence that is sufficiently long in a yet weaker sense and so forth.
The iteration goes for   four times and in the last iteration, to be sufficiently long 
means something like $r>n_0$. Roughly speaking, this can be seen as 
a sufficiently long notion of cardinal $4$.
On the other hand, we don't know if a sufficiently long notion of cardinal $1$
suffices. i.e.,

\begin{question}

 Is it true that for some $r\in\omega$, for every $n>0$,
$\opp_2^2(2 \nr)$ does not hold;
actually we wonder whether $\opp_2^2(2nnn)$ is not true
for all $n>0$.
More generally, is it true that for every $\h n>0$, there exists an
$r\in\omega$ such that
for every $n>0$,
$\opp_2^2(\h n \nr)$ does not hold.

\end{question}

For $n=2$, although we don't know how to prove,
 Adam P. Goucher, using SAT solver,
showed that $\opp_2^2(2222)$ is not true

(
\href{https://mathoverflow.net/questions/293112/ramsey-type-theorem}
{https://mathoverflow.net/questions/293112/ramsey-type-theorem}
).
Note that $\opp_2^2(222)$ is true (see Proposition \ref{combprop1}).

\bibliographystyle{amsplain}
\bibliography{F:/6+1/Draft/bibliographylogic}

\end{document}